\documentclass[reqno]{amsart}

\usepackage[usenames,dvipsnames,svgnames,table]{xcolor}
\usepackage{IEEEtrantools}

\usepackage{amssymb,latexsym,amsfonts,amsmath}
\usepackage{graphicx}
\usepackage{eso-pic}
\usepackage{epsfig}
\usepackage{graphicx}
\usepackage{mathtools}
\usepackage{tikz}
\usetikzlibrary{automata}
\usetikzlibrary{shapes}
\usetikzlibrary{calc,shapes,arrows}
\usepackage{amssymb,latexsym,amsfonts,amsmath}
\usepackage{graphicx}
\usepackage{mdwlist}
\usepackage{refcount}
\usepackage{comment}

\newtheorem{theorem}{Theorem}[section]
\newtheorem{lemma}[theorem]{Lemma}

\newtheorem{proposition}[theorem]{Proposition}
\newtheorem{corollary}[theorem]{Corollary}

\newtheorem{definition}[theorem]{Definition}

\newtheorem{remark}[theorem]{Remark}
\numberwithin{equation}{section}
\newtheorem{assumption}{Assumption}

\xdefinecolor{MATLABblue}{rgb}{0.0,0.45,.74}
\xdefinecolor{MATLABred}{rgb}{0.85,0.33,.1}

\newcommand{\R}{{\mathbb{R}}}

\newcommand{\N}{{\mathbb{N}}}

\newcommand{\intcc}[1]{\ensuremath{{\left[#1\right]}}}
\newcommand{\intoc}[1]{\ensuremath{{\left]#1\right]}}}
\newcommand{\intco}[1]{\ensuremath{{\left[#1\right[}}}
\newcommand{\intoo}[1]{\ensuremath{{\left]#1\right[}}}
\newcommand{\as}{\overset{\ra}{s}}
\newcommand{\st}{{\rm s.t.}}
\newcommand{\KL}{\mathcal{KL}}
\newcommand{\Kinf}{\mathcal{K}_\infty}
\newcommand{\Rp}{\R_{\ge0}}
\newcommand{\eps}{\varepsilon}

\newcommand{\da}{\downarrow}

\usepackage[normalem]{ulem}

\usepackage{graphicx,keyval,trig}
\usepackage{amsmath,amssymb,amsfonts,amssymb,dsfont}
\usepackage{mathtools, mathrsfs}

\DeclareMathOperator{\diff}{d}

\DeclareMathOperator{\im}{im}
\DeclareMathOperator{\ke}{ker}

\newcommand{\ra}{\rightarrow}
\newcommand{\sigalg}{\mathcal{F}}
\newcommand{\filtration}{\mathds{F}}

\newcommand{\ol}{\overline}
\newcommand{\Let}{:=}
\newcommand{\EE}{\mathds{E}}
\newcommand{\PP}{\mathds{P}}

\begin{document}

\begin{abstract}
In this paper we propose a compositional framework for the construction of approximations of the interconnection of a class of stochastic hybrid systems. As special cases, this class of systems includes both jump linear stochastic systems and linear stochastic hybrid automata. In the proposed framework, an approximation is itself a stochastic hybrid system, which can be used as a replacement of the original stochastic hybrid system in a controller design process. We employ a notion of so-called stochastic simulation function to quantify the error between the approximation and the original system. In the first part of the paper, we derive sufficient conditions which facilitate the compositional quantification of the error between the interconnection of stochastic hybrid subsystems and that of their approximations using the quantified error between the stochastic hybrid subsystems and their corresponding approximations. In particular, we show how to construct stochastic simulation functions for approximations of interconnected stochastic hybrid systems using the stochastic simulation function for the approximation of each component. In the second part of the paper, we focus on a specific class of stochastic hybrid systems, namely, jump linear stochastic systems, and propose a constructive scheme to determine approximations together with their stochastic simulation functions for this class of systems. Finally, we illustrate the effectiveness of the proposed results by constructing an approximation of the interconnection of four jump linear stochastic subsystems in a compositional way.
\end{abstract}

\title[Approximations of Stochastic Hybrid Systems: A Compositional Approach]{Approximations of Stochastic Hybrid Systems: A Compositional Approach}

\author[M. Zamani]{Majid Zamani$^1$} 
\author[M. Rungger]{Matthias Rungger$^1$}
\author[P. Mohajerin Esfahani]{Peyman Mohajerin Esfahani$^2$} 
\address{$^1$Department of Electrical and Computer Engineering, Technische Universit\"at M\"unchen, D-80290 Munich, Germany.}
\email{\{zamani,matthias.rungger\}@tum.de}
\urladdr{http://www.hcs.ei.tum.de}
\address{$^2$Automatic Control Laboratory at ETH Z\"{u}rich and the Risk Analytics and Optimization Chair at ETH Lausanne, Switzerland.}
\email{mohajerin@control.ee.ethz.ch}
\urladdr{http://control.ee.ethz.ch/$\sim$peymanm}

\maketitle

\section{Introduction}
Stochastic hybrid systems are a general class of dynamical systems consisting of continuous and discrete dynamics subject to probabilistic noise and events. In the past few years, this class of systems has become ubiquitous in many different fields due to the need for a rigorous modeling framework for many safety-critical applications. Examples of those applications include air traffic control \cite{glover}, biochemistry \cite{singh}, communication networks \cite{hespanha}, and systems biology \cite{hu}. The design of controllers to enforce certain given complex specifications, e.g. those expressed via formulae in linear temporal logic (LTL) \cite{katoen08}, in a reliable and cost effective way is a grand challenge in the study of many of those safety-critical applications. One promising direction to achieve those objectives is the use of simpler (in)finite approximations of the given systems as a replacement in the controller design process. Those approximations allow us to design controllers for them and then refine the controllers to the ones for the concrete complex systems, while provide us with the quantified errors in this detour controller synthesis scheme. 

In the past few years there have been several results on the (in)finite approximations of {\em continuous-time} stochastic (hybrid) systems. Existing results include the construction of finite approximations for 
stochastic dynamical systems under contractivity assumptions \cite{abate}, restricted to models with no control inputs, a finite Markov decision process approximation of a linear stochastic control system \cite{LAB09}, however without a quantitative relationship between approximation and concrete model, and the construction of finite bisimilar abstractions for stochastic control systems \cite{majid8,majid10}, for stochastic switched systems \cite{majid18}, for randomly switched stochastic systems \cite{majid11}, and the construction of sound finite abstractions for stochastic control systems without any stability property \cite{majid7}. Further, the results in \cite{julius1} check the relationship between infinite approximations and a given class of stochastic hybrid systems via a notion of stochastic (bi)simulation functions. However, the results in \cite{julius1} do not provide any approximations and moreover appear to be computationally intractable in the case of systems with inputs because one requires to solve a game in order to quantify the approximation error. Note that all the proposed results in
\cite{abate,LAB09,majid8,majid10,majid18,majid11,majid7,julius1} take a monolithic view of continuous-time stochastic (hybrid) systems, where the entire system is approximated. This monolithic view interacts
badly with the construction of approximations, whose complexity grows (possibly exponentially) in the number of continuous state variables in the model.

In this paper, we provide a compositional framework for the construction
of \emph{infinite} approximations of the interconnection of a class of stochastic hybrid
systems, in which the continuous dynamics are modeled by stochastic differential equations
and the switches are modeled as Poisson processes. As special cases, this class of systems includes both jump linear stochastic systems (JLSS) and
linear stochastic hybrid automata \cite{julius1}. Our approximation framework is based on a new notion of stochastic
simulation functions. In this framework, an approximation, which is itself a stochastic hybrid system (potentially with lower dimension and simpler interconnection topology), acts as a substitute in the controller design process. The stochastic simulation function
is used to quantify the error in this detour controller synthesis scheme. Although an approximation in our framework might not be directly amenable to
algorithmic synthesis methods based on automata-theoretic concepts \cite{MalerPnueliSifakis95} which require finite approximations, our approach facilitates the construction
of potentially lower-dimensional less-interconnected stochastic hybrid systems as approximations and, hence, can be interpreted as the first pre-processing step in
the construction of a finite approximation.

In the first part of the paper, we derive sufficient small-gain type conditions, similar to the ones in~\cite{DIW11}, under which one can quantify the error between the interconnection of stochastic hybrid subsystems and that of their approximations in a compositional way by using the errors between stochastic hybrid subsystems and their approximations. 
In the second part of the paper, we focus on JLSS and propose a computational scheme to construct infinite approximations of this class of systems, together with the corresponding stochastic simulation functions. To show the effectiveness of the proposed results, we construct an approximation (two disjoint 3 dimensional JLSS) of the interconnection of four JLSS (overall 10 dimensions) in a compositional way and then use the approximation in order to design a safety controller for the original interconnected system. Note that the controller synthesis would not have been possible without the use of the approximation.

The recent work in \cite{majid17} provides a compositional scheme for the construction of infinite approximations of interconnected deterministic control systems without any hybrid dynamic. The results in this paper are complementary to the ones in \cite{majid17} as we extend our focus to the class of stochastic hybrid systems. A preliminary investigation of our results on the compositional construction of infinite approximations of interconnected stochastic hybrid systems appeared in \cite{majid14}. In this paper we
present a detailed and mature description of the results announced in \cite{majid14}, including proposing a new notion of stochastic simulation functions which is computationally more tractable in the case of systems with inputs and providing constructive means to
compute approximations of JLSS.

\section{Stochastic Hybrid Systems}
\subsection{Notation}
We denote by $\N$ the set of nonnegative integer numbers and by
$\R$ the set of real numbers.
We annotate those symbols with subscripts to restrict them in
the obvious way, e.g. $\R_{>0}$ denotes the positive real numbers. The symbols $I_n$, $0_n$, and $0_{n\times{m}}$ denote the identity matrix, zero vector, and zero matrix in $\R^{n\times{n}}$, $\R^n$, and $\R^{n\times{m}}$, respectively. For $a,b\in\R$ with $a\le b$, we denote the closed, open, and half-open intervals in $\R$ by $\intcc{a,b}$,
$\intoo{a,b}$, $\intco{a,b}$, and $\intoc{a,b}$, respectively. For $a,b\in\N$ and $a\le b$, we
use $\intcc{a;b}$, $\intoo{a;b}$, $\intco{a;b}$, and $\intoc{a;b}$ to
denote the corresponding intervals in $\N$.
Given $N\in\N_{\geq1}$, vectors $x_i\in\R^{n_i}$, $n_i\in\N_{\geq1}$ and $i\in\intcc{1;N}$, we
use $x=[x_1;\ldots;x_N]$ to denote the vector in $\R^n$ with
$n=\sum_{i=1}^N n_i$. Similarly, we use $X=[X_1;\ldots;X_N]$ to denote the matrix in $\R^{n\times m}$ with
$n=\sum_{i=1}^N n_i$, given $N\in\N_{\geq1}$, matrices $X_i\in\R^{n_i\times m}$, $n_i\in\N_{\geq1}$, and $i\in\intcc{1;N}$. Given a vector \mbox{$x\in\mathbb{R}^{n}$}, we denote by $\Vert x\Vert$ the Euclidean norm of $x$. The distance of a point $x\in\R^n$ to a set $\mathsf{D}\subseteq\R^n$ is defined as $\Vert x\Vert_{\mathsf{D}}=\inf_{d\in \mathsf{D}}\Vert x-d\Vert$. Given a matrix $P=\{p_{ij}\}\in\R^{n\times{n}}$, we denote by $\text{Tr}({P})=\sum_{i=1}^np_{ii}$ the trace of $P$. 

Given a function $f:\R^n\rightarrow \R^m$ and $\bar x\in\R^m$, we use
$f\equiv \bar x$ to denote that $f(x)=\bar x$ for all $x\in\R^n$. If
$x$ is the zero vector, we simply write $f\equiv 0$. Given a function \mbox{$f:\mathbb{R}_{\geq0}\rightarrow\mathbb{R}^n$}, the (essential) supremum of $f$ is denoted by $\Vert f\Vert_{\infty} \Let \text{(ess)sup}\{\Vert f(t)\Vert,t\geq0\}$. Measurability throughout this paper refers to Borel measurability. A continuous function \mbox{$\gamma:\mathbb{R}_{\geq0}\rightarrow\mathbb{R}_{\geq0}$}, is said to belong to class $\mathcal{K}$ if it is strictly increasing and \mbox{$\gamma(0)=0$}; $\gamma$ is said to belong to class $\mathcal{K}_{\infty}$ if \mbox{$\gamma\in\mathcal{K}$} and $\gamma(r)\rightarrow\infty$ as $r\rightarrow\infty$. A continuous function $\beta:\mathbb{R}_{\geq0}\times\mathbb{R}_{\geq0}\rightarrow\mathbb{R}_{\geq0}$ is said to belong to class $\mathcal{KL}$ if, for each fixed $t$, the map $\beta(r,t)$ belongs to class $\mathcal{K}$ with respect to $r$ and, for each fixed nonzero $r$, the map $\beta(r,t)$ is decreasing with respect to $t$ and $\beta(r,t)\rightarrow 0$ as \mbox{$t\rightarrow\infty$}.

\subsection{Stochastic hybrid systems}
Let $(\Omega, \sigalg, \PP)$ be a probability space endowed with a filtration $\filtration = (\sigalg_s)_{s\geq 0}$ satisfying the usual conditions of completeness and right continuity \cite[p.\ 48]{ref:KarShr-91}. Let $(W_s)_{s \ge 0}$ be a $\widetilde{p}$-dimensional $\filtration$-Brownian motion and $(P_s)_{s\ge0}$ be a $\widetilde{q}$-dimensional $\filtration$-Poisson process. We assume that the Poisson process and the Brownian motion are independent of each other. The Poisson process $P_s\Let\left[P_s^1;\ldots;P_s^{\widetilde q}\right]$ model $\widetilde{q}$ kinds of events whose occurrences are assumed to be independent of each other.

\begin{definition}
\label{Def_control_sys}The class of \emph{stochastic hybrid systems} with which we deal in this paper is the tuple $\Sigma=\left(\R^n,\R^m,\R^p,\mathcal{U},\mathcal{W},f,\sigma,r,\R^q,h\right)$, where $\R^n$, $\R^m$, $\R^p$, and $\R^q$ are the state, external input, internal input, and output spaces, respectively, and
\begin{itemize}
\item $\mathcal{U}$ is a subset of the set of all $\filtration$-progressively measurable processes with values in $\R^m$; see \cite[Def. 1.11]{ref:KarShr-91}; 
\item $\mathcal{W}$ is a subset of the set of all $\filtration$-progressively measurable processes with values in $\R^p$; 
\item $f:\R^n\times\R^m\times\R^p\rightarrow\R^n$ is the drift term which is globally Lipschitz continuous: there exist constants $L_x,L_u,L_w\in\R_{\geq0}$ such that: $\Vert f(x,u,w)-f(x',u',w')\Vert\leq L_x\Vert x-x'\Vert +  L_u\Vert u-u'\Vert+L_w\Vert w-w'\Vert$ for all $x,x'\in\R^n$, all $u,u'\in\R^m$, and all $w,w'\in\R^p$;
\item $\sigma:\R^n\rightarrow\R^{n\times{\widetilde{p}}}$ is the diffusion term which is globally Lipschitz continuous;
\item $r:\R^n\rightarrow\R^{n\times{\widetilde q}}$ is the reset function which is globally Lipschitz continuous;
\item $h:\R^n\rightarrow\R^q$ is the output map. 
\end{itemize}
\end{definition}

A stochastic hybrid system $\Sigma$ satisfies
\begin{align}\label{eq0}
\Sigma:\left\{\hspace{-2mm}\begin{array}{l}\diff \xi(t)=f(\xi(t),\nu(t),\omega(t))\diff t+\sigma(\xi(t))\diff W_t+r(\xi(t))\diff P_t,\\\hspace{2mm}\zeta(t)=h(\xi(t)),\end{array}\right.
\end{align}
$\PP$-almost surely ($\PP$-a.s.) for any $\nu\in\mathcal{U}$ and any $\omega\in\mathcal{W}$, where stochastic process \mbox{$\xi:\Omega \times \R_{\geq0} \rightarrow \R^n$} is called a \textit{solution process} of $\Sigma$ and stochastic process \mbox{$\zeta:\Omega \times \R_{\geq0} \rightarrow \R^q$} is called an \textit{output trajectory} of $\Sigma$. We call the tuple $(\xi,\zeta,\nu,\omega)$ a \emph{trajectory}
of $\Sigma$, consisting of a solution process $\xi$, an output
trajectory $\zeta$, and input trajectories
$\nu$ and $\omega$, that satisfies \eqref{eq0} $\PP$-a.s..
We also write $\xi_{a \nu\omega}(t)$ to denote the value of the solution process at time $t\in\R_{\geq0}$ under the input trajectories $\nu$ and $\omega$ from initial condition $\xi_{a \nu\omega}(0) = a$ $\PP$-a.s., in which $a$ is a random variable that is $\sigalg_0$-measurable. We denote by $\zeta_{a\nu\omega}$ the output trajectory corresponding to the solution process $\xi_{a \nu\omega}$. Here, we assume that the Poisson processes $P_s^i$, for any $i\in[1;\widetilde q]$, have the rates of $\lambda_i$. We emphasize that the postulated assumptions on $f$, $\sigma$, and $r$ ensure existence, uniqueness, and strong Markov property of the solution processes \cite{borkar1989}. 

\begin{remark}
We refer the interested readers to Section IV in \cite{julius1} showing how one can cast linear stochastic hybrid automata (LSHA) as jump linear stochastic systems (JLSS) (c.f. Section \ref{s:lin}) which are a specific class of the ones introduced in Definition \ref{Def_control_sys}. 
\end{remark}

\section{Stochastic Simulation Function}
Before introducing the notion of stochastic simulation functions, we first need to define the infinitesimal generator of a stochastic process.
\begin{definition}
Let
$\Sigma=(\R^n,\R^m,\R^p,\mathcal{U},\mathcal{W},f,\sigma,r,\R^q,h)$ and
$\hat\Sigma=(\R^{\hat n},\R^{\hat m},\R^{\hat p},\hat{\mathcal{U}},\hat{\mathcal{W}},\hat f,\hat\sigma,\hat r,\R^{\hat q},\hat h)$ be
two stochastic hybrid systems with solution processes $\xi$ and $\hat\xi$, respectively. Consider a twice continuously differentiable function \mbox{$V:\R^{n}\times \R^{\hat n}\to\R_{\ge0}$}. The infinitesimal generator of the stochastic process $\Xi=[\xi;\hat{\xi}]$, denoted by $\mathcal{L}$, acting on function $V$ is defined in \cite[Section 1.3]{ref:Oks-jump} as:
 \begin{align}\label{generator}
	\mathcal{L} V\left(x, \hat{x}\right)\Let& \left[\partial_xV~~\partial_{\hat{x}}V\right] \begin{bmatrix} f(x,u,w)\\\hat{f}\left(\hat{x},\hat{u},\hat w\right)\end{bmatrix}+\frac{1}{2} \text{Tr} \left(\begin{bmatrix} \sigma(x) \\ \hat\sigma\left(\hat{x}\right) \end{bmatrix}\left[\sigma^T(x)~~\hat\sigma^T\left(\hat{x}\right)\right] \begin{bmatrix}
\partial_{x,x} V & \partial_{x,\hat{x}} V \\ \partial_{\hat{x},x} V & \partial_{\hat{x},\hat{x}} V
\end{bmatrix}\right)\\\notag&+\sum_{i=1}^{\widetilde q}\lambda_i\left(V\left(x+r(x)\mathsf{e}_i,\hat{x}+\hat{r}(\hat{x})\mathsf{e}_i\right)-V(x,\hat{x})\right),
\end{align} 
for every $x\in\R^n$, $\hat x\in\R^{\hat n}$, $u\in\R^m$, $\hat u\in\R^{\hat m}$, $w\in\R^p$, and $\hat w\in\R^{\hat p}$.
\end{definition}

Now, we introduce a notion of stochastic simulation functions, inspired by the notion of simulation function
in~\cite{majid17}, for deterministic control systems distinguishing the role of internal and external inputs.

\begin{definition}\label{d:sf} Let
$\Sigma=(\R^n,\R^m,\R^p,\mathcal{U},\mathcal{W},f,\sigma,r,\R^q,h)$ and
$\hat\Sigma=(\R^{\hat n},\R^{\hat m},\R^p,\hat{\mathcal{U}},\mathcal{W},\hat f,\hat\sigma,\hat r,\R^q,\hat h)$ be
two stochastic hybrid systems with the same internal input and output space
dimension. A twice continuously differentiable 
function \mbox{$V:\R^{n}\times \R^{\hat n}\to\R_{\ge0}$} is
called a \emph{stochastic simulation function from $\hat\Sigma$ to $\Sigma$ in the $k$th moment} (SSF-M$_k$), where $k\geq1$, if it has polynomial growth rate and 
for any $x\in\R^n$ and $\hat x\in\R^{\hat n}$ one has
\begin{IEEEeqnarray}{c}\label{e:sf:1}
\alpha(\Vert h(x)-\hat h(\hat x)\Vert^k)\le V(x,\hat x),
\end{IEEEeqnarray}
and $\forall \hat u\in\R^{\hat m}$ $\forall\hat w\in\R^p$ $\exists u\in\R^m$ $\forall w\in\R^p$ one obtains
\begin{align}\label{inequality1} 
	\mathcal{L} V\left(x, \hat{x}\right)	 \leq-\eta (V\left(x,\hat{x}\right))+\rho_{\mathrm{ext}}(\Vert\hat
u\Vert^k)+\rho_{\mathrm{int}}(\Vert w-\hat w\Vert^k),
\end{align} 
for some $\mathcal{K}_\infty$ functions
$\alpha,\eta,\rho_{\mathrm{ext}},\rho_{\mathrm{int}}$, where $\mathsf{e}_i\in\R^{\widetilde q}$ denotes the vector with $1$ in the $i$th coordinate and $0$'s elsewhere, $\alpha,\eta$ are convex functions, and $\rho_{\mathrm{ext}},\rho_{\mathrm{int}}$ are concave ones.

\end{definition}

In the above definition, the symbols $\partial_x$ and $\partial_{x,\hat{x}}$ denote the first and the second order partial derivatives with respect to $x$ and $x$ and $\hat{x}$, respectively.

We say that a stochastic hybrid system $\hat\Sigma$ is \emph{approximately alternatingly simulated in the $k$th moment} by a stochastic hybrid
system $\Sigma$ or $\Sigma$ \emph{approximately alternatingly simulates in the $k$th moment}
$\hat\Sigma$, denoted by $\hat\Sigma\preceq_{\mathcal{AS}}^k\Sigma$, if
there exists a SSF-M$_k$ function $V$ from $\hat\Sigma$ to
$\Sigma$. We call $\hat \Sigma$ an
\emph{abstraction} of $\Sigma$. 


\begin{remark}
Note that the notion of SSF-M$_k$ here is different from the notion of stochastic simulation function in \cite[Definition 2]{julius1} requiring the existences of a supermartingale function \cite[Appendix C]{oksendal} whose construction is computationally intractable in the case of (even linear) systems with inputs because one requires to solve a game to compute this function. On the other hand, the notion of stochastic (bi)simulation function in \cite{julius1} is stronger than the notion of SSF-M$_k$ as it provides a lower bound on the probability of satisfaction of specifications for which satisfiability can be obtained at all time instances rather than for a bounded time horizon (cf. Proposition \ref{lemma6}) or at single time instances (cf. Proposition \ref{lemma7}). We refer the interested readers to Subsection V.B in \cite{majid8} for more detailed information about those differences in satisfiability.
\end{remark}

\begin{remark}
If the drift, diffusion, and reset terms in $\Sigma$ and $\hat\Sigma$ in Definition \ref{d:sf} are polynomial, one can use some sum of squares based semidefinite programing tools, such as SOSTOOLS \cite[Subsection 4.2]{antonis}, in order to efficiently search for a (sum of squares) SSF-M$_k$ function from $\hat\Sigma$ to $\Sigma$ which may not exist in general.
\end{remark}

The following theorem shows the importance of the existence of a SSF-M$_k$ function by quantifying the error between the behaviors of $\Sigma$ and the ones of its abstraction $\hat \Sigma$.
\begin{theorem}\label{theorem1}
Let 
$\Sigma=(\R^n,\R^m,\R^p,\mathcal{U},\mathcal{W},f,\sigma,r,\R^q,h)$ and
$\hat\Sigma=(\R^{\hat n},\R^{\hat m},\R^p,\hat{\mathcal{U}},\mathcal{W},\hat f,\hat\sigma,\hat r,\R^q,\hat h)$.
Suppose $V$ is an SSF-M$_k$ function from $\hat\Sigma$ to $\Sigma$. Then, there
exist a $\mathcal{KL}$ function $\beta$ and $\mathcal{K}_\infty$ functions
$\gamma_{\mathrm{ext}}$, $\gamma_{\mathrm{int}}$ such that 
for any
$\hat\nu\in\hat{\mathcal{U}}$, any $\hat \omega \in\mathcal{W}$, and any random variable $a$ and $\hat{a}$ that are $\sigalg_0$-measurable\footnote{Note that $\sigalg_0$ may be the trivial sigma-algebra, i.e., 
$a$ and $\hat{a}$ are deterministic initial conditions.}, there exists $\nu\in{\mathcal{U}}$ such that for all $\omega\in\mathcal{W}$ the following inequality holds:
\begin{align}
\label{inequality} 
	\EE[\Vert\zeta_{a\nu\omega}(t) -\hat \zeta_{\hat a \hat \nu\hat\omega}(t)\Vert^k] \le \beta\left(\EE[V(a,\hat a)],t\right ) 
	+ \gamma_{\mathrm{ext}}(\EE[\Vert\hat \nu\Vert^k_\infty]) + \gamma_{\mathrm{int}}(\EE[\Vert\omega-\hat \omega\Vert^k_\infty]).
\end{align} 
\end{theorem}

The proof of Theorem \ref{theorem1} requires the following preparatory lemma and is provided in the Appendix.

\begin{lemma}
\label{lem:KL}
	Let $g$ be a non-negative constant and $\eta$ be a $\Kinf$ function. Suppose that the function $y:\Rp \ra \Rp$ is continuous and we have $y(t) \le y(t_0) + \int_{t_0}^{t} [-\eta\big(y(\tau)\big)+g] \diff \tau$ for all $t \ge t_0 \ge 0$. Then, there exists a $\KL$ function $\vartheta$ such that 
	$$y(t) \le \max\Big\{\vartheta\big(y(0),t\big), \eta^{-1}\big(2g\big)\Big\}, \quad \forall t \ge 0.$$
\end{lemma}

The proof of Lemma \ref{lem:KL} is provided in the Appendix. 

Note that the importance of the result provided in Theorem~\ref{theorem1} is that one can synthesize a controller
for the abstraction $\hat \Sigma$, which is potentially easier (e.g., lower dimension and simpler interconnection topology) to enforce some complex specification, for example given in LTL. Then there exists a controller for the concrete stochastic hybrid system $\Sigma$ satisfying the same complex specification. The error, introduced in the design process by taking the detour through
the abstraction, is quantified by inequality~\eqref{inequality}. In Section \ref{s:lin}, we show how one can actually \emph{refine} a controller designed for the abstract JLSS to a controller for the original JLSS via a so-called \emph{interface} function.

%

The notion of stochastic simulation function in this work can also be used to lower bound the probability that the Euclidean distance between any output trajectory of the abstract model and the corresponding one of the concrete model remains close. 

We make the above statement more precise with the following results.

\begin{proposition}\label{lemma6}
Let $\Sigma$ and
$\hat\Sigma$ be
two stochastic hybrid systems with the same internal input and output space
dimension. Suppose $V$ is an SSF-M$_k$ function from $\hat\Sigma$ to $\Sigma$ and the $\mathcal{K}_\infty$ function $\eta$ in \eqref{inequality1} satisfies $\eta(r)\geq\theta r$ for some $\theta\in\R_{>0}$ and any $r\in\R_{\geq0}$. For any
$\hat\nu\in\hat{\mathcal{U}}$, any $\hat \omega \in\mathcal{W}$, and any random variable $a$ and $\hat{a}$ that are $\sigalg_0$-measurable, there exists $\nu\in{\mathcal{U}}$ such that for all $\omega\in\mathcal{W}$ the following inequalities \eqref{simultaneous1} and \eqref{simultaneous2} hold provided that there exists a constant $\epsilon\geq0$ satisfying $\epsilon\geq\rho_{\mathrm{ext}}(\Vert\hat \nu\Vert^k_\infty)+\rho_{\mathrm{int}}(\Vert \omega-\hat \omega\Vert^k_\infty)$:
\begin{align}\label{simultaneous1}
&\PP\left\{\sup_{0\leq t\leq T}\Vert\zeta_{a\nu\omega}(t)-\hat \zeta_{\hat a\hat \nu\hat\omega}(t)\Vert\geq\varepsilon\,\,|\,\,[a;\hat a]\right\}\leq1-\left(1-\frac{V(a,\hat a)}{\alpha\left(\varepsilon^k\right)}\right)\mathsf{e}^{-\frac{\epsilon T}{\alpha\left(\varepsilon^k\right)}},~~~\text{if}~\alpha\left(\varepsilon^k\right)\geq\frac{\epsilon}{\theta},\\\label{simultaneous2}
&\PP\left\{\sup_{0\leq t\leq T}\Vert\zeta_{a\nu\omega}(t)-\hat \zeta_{\hat a\hat \nu\hat\omega}(t)\Vert\geq\varepsilon\,\,|\,\,[a;\hat a]\right\}\leq\frac{\theta V(a,\hat a)+\left(\mathsf{e}^{T\theta}-1\right)\epsilon}{\theta\alpha\left(\varepsilon^k\right)\mathsf{e}^{T\theta}},~~~~~~~~\text{if}~\alpha\left(\varepsilon^k\right)\leq\frac{\epsilon}{\theta}.
\end{align}
\end{proposition}

The proof of Proposition \ref{lemma6} is provided in the Appendix.

As an alternative to the previous result, we now use the notion of stochastic simulation function to lower bound the probability of the Euclidean distance between any output trajectory of the abstract model and the corresponding one of the concrete model point-wise in time: this error bound is sufficient
to work with those specifications for which satisfiability can be
achieved at single time instances, such as next ($\bigcirc$) and
eventually ($\Diamond$) in LTL. Please look at the explanation after the proof
of Proposition 5.11 in \cite{majid8} for more details.

\begin{proposition}\label{lemma7}
Let $\Sigma$ and
$\hat\Sigma$ be
two stochastic hybrid systems with the same internal input and output space
dimension. Suppose $V$ is an SSF-M$_k$ function from $\hat\Sigma$ to $\Sigma$. For any
$\hat\nu\in\hat{\mathcal{U}}$, any $\hat \omega \in\mathcal{W}$, and any random variable $a$ and $\hat{a}$ that are $\sigalg_0$-measurable, there exists $\nu\in{\mathcal{U}}$ such that for all $\omega\in\mathcal{W}$ the following inequality holds for all $t\in\R_{\geq0}$:
\begin{align}\label{instances}
&\PP\left\{\Vert\zeta_{a\nu\omega}(t)-\hat \zeta_{\hat a\hat \nu\hat\omega}(t)\Vert\geq\varepsilon\right\}\leq\frac{\left(\beta\left(\EE[V(a,\hat a)],t\right )+ \gamma_{\mathrm{ext}}(\EE[\Vert\hat \nu\Vert^k_\infty]) + \gamma_{\mathrm{int}}(\EE[\Vert\omega-\hat \omega\Vert^k_\infty])\right)^{\frac{1}{k}}}{\varepsilon},
\end{align}
where $\beta$, $\gamma_{\mathrm{ext}}$, and $\gamma_{\mathrm{int}}$ are the functions appearing in \eqref{inequality}.
\end{proposition}

The proof of Proposition \ref{lemma7} is provided in the Appendix.

In the next section, we work with interconnected stochastic hybrid systems \emph{without}
internal inputs, resulting from the interconnection of stochastic hybrid subsystems having both internal and external signals. In this case, the interconnected stochastic hybrid systems reduce
to the tuple $\Sigma=\left(\R^n,\R^m,\mathcal{U},f,\sigma,r,\R^q,h\right)$ and the drift term becomes $f:\R^n\times\R^m\to\R^n$. 
In this view, inequality \eqref{inequality1} is not quantified
over $w,\hat w\in\R^p$, and, hence, the term $\rho_{\mathrm{int}}(\Vert w-\hat w\Vert^k)$ is omitted as well.
Similarly, the results in Theorem~\ref{theorem1} and Propositions \ref{lemma6} and \ref{lemma7}
are modified accordingly, i.e., for systems without internal inputs the
inequalities~\eqref{inequality}, \eqref{simultaneous1}, \eqref{simultaneous2}, and \eqref{instances} are not quantified
over $\omega,\hat \omega\in\mathcal{W}$ and, hence, the term
$\gamma_{\mathrm{int}}(\EE[\Vert\omega-\hat \omega\Vert^k_\infty])$ is omitted in inequalities \eqref{inequality} and \eqref{instances} and $\epsilon$ is lower bounded as $\epsilon\geq\rho_{\mathrm{ext}}(\Vert\hat \nu\Vert^k_\infty)$ in Proposition \ref{lemma6} as well.

The next corollary provides a similar result as the one of Proposition \ref{lemma6} but by considering an infinite time horizon and interconnected stochastic hybrid systems and assuming $\hat\nu\equiv0$, resulting in $\epsilon=0$. The relation proposed in this corollary recovers the one proposed in \cite[Theorem 3]{julius1}.

\begin{corollary}\label{lemma8}
Let $\Sigma$ and
$\hat\Sigma$ be
two interconnected stochastic hybrid systems with the same output space
dimension. Suppose $V$ is an SSF-M$_k$ function from $\hat\Sigma$ to $\Sigma$. For
$\hat\nu\equiv0$ and any random variable $a$ and $\hat{a}$ that are $\sigalg_0$-measurable, there exists $\nu\in{\mathcal{U}}$ such that the following inequality holds:
\begin{align}\nonumber
\PP\left\{\sup_{0\leq t< \infty}\Vert\zeta_{a\nu}(t)-\hat \zeta_{\hat a 0}(t)\Vert>\varepsilon\,\,|\,\,[a;\hat a]\right\}\leq\frac{V(a,\hat a)}{\alpha\left(\varepsilon^k\right)}.
\end{align}
\end{corollary}

The proof of Corollary \ref{lemma8} is provided in the Appendix.

Note that under the assumptions of Corollary \ref{lemma8} any SSF-M$_k$ function is also a stochastic simulation function as in \cite{julius1}.

\section{Compositionality Result}
\label{s:inter}

In this section, we analyze interconnected stochastic hybrid systems and show
how to construct an abstraction of an
interconnected stochastic hybrid system together with the corresponding stochastic simulation function. The
definition of the interconnected stochastic hybrid system is based on the notion of
interconnected systems introduced in~\cite{TI08}. 

\subsection{Interconnected stochastic hybrid systems}

We consider $N\in\N_{\geq1}$ stochastic hybrid subsystems
\begin{IEEEeqnarray*}{c'c}
\Sigma_i=\left(\R^{n_i},\R^{m_i},\R^{p_i},\mathcal{U}_i,\mathcal{W}_i,f_i,\sigma_i,r_i,\R^{q_i},h_i\right),
&i\in\intcc{1;N}
\end{IEEEeqnarray*}
with partitioned internal inputs and outputs
\begin{align}\notag
w_i&=\intcc{w_{i1};\ldots;w_{i(i-1)};w_{i(i+1)};\ldots;w_{iN}},~ w_{ij}\in\R^{p_{ij}}\\\label{e:iop}
y_i&=\intcc{y_{i1};\ldots;y_{iN}},~ y_{ij}\in\R^{q_{ij}}
\end{align}
and output function
\begin{IEEEeqnarray}{c'c}\label{e:of}
h_i(x_i)=\intcc{h_{i1}(x_i);\ldots;h_{iN}(x_i)},
\end{IEEEeqnarray}
as depicted schematically in Figure \ref{system}.
\begin{figure}[ht]
\begin{center}
\includegraphics[width=6cm]{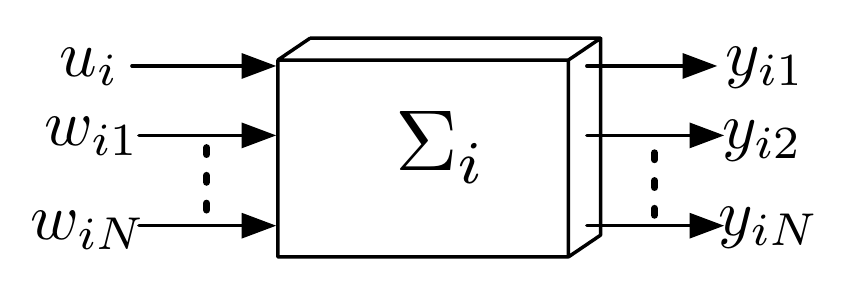}
\caption{Input/output configuration of stochastic hybrid subsystem~$\Sigma_i$.}
\label{system}
\end{center}
\end{figure}

We interpret the outputs $y_{ii}$ as \emph{external} ones, whereas the outputs $y_{ij}$ with $i\neq j$ are \emph{internal}
ones which are used to define the
interconnected stochastic hybrid systems. In particular, we assume that the
dimension of $w_{ij}$ is equal to the dimension of $y_{ji}$, i.e.,
the following \emph{interconnection constraints} hold:
\begin{IEEEeqnarray}{c'c}\label{e:ic}
p_{ij}=q_{ji},&\forall i,j\in\intcc{1;N},~i\neq j.
\end{IEEEeqnarray}
If there is no connection from stochastic hybrid subsystem $\Sigma_{i}$ to
$\Sigma_j$, then we assume that the connecting output function is identically
zero for all arguments, i.e., $h_{ij}\equiv 0$. We define the \emph{interconnected stochastic hybrid system} as the following.

\begin{definition}\label{d:ics}
Consider $N\in\N_{\geq1}$ stochastic hybrid subsystems $
\Sigma_i=\left(\R^{n_i},\R^{m_i},\R^{p_i},\mathcal{U}_i,\mathcal{W}_i,f_i,\sigma_i,r_i,\R^{q_i},h_i\right)$,
$i\in\intcc{1;N}$, with the input-output configuration given
by~\eqref{e:iop}-\eqref{e:ic}. The \emph{interconnected stochastic hybrid
system} $\Sigma=\left(\R^{n},\R^{m},\mathcal{U},f,\sigma,r,\R^{q},h\right)$, denoted by
$\mathcal{I}(\Sigma_1,\ldots,\Sigma_N)$, follows by $n=\sum_{i=1}^Nn_i$,
$m=\sum_{i=1}^Nm_{i}$, $q=\sum_{i=1}^Nq_{ii}$, and functions
\begin{align*}
f(x,u)&\Let\intcc{f_1(x_1,u_1,w_1);\ldots;f_N(x_N,u_N,w_N)},\\
\sigma(x)&\Let\intcc{\sigma_1(x_1);\ldots;\sigma_N(x_n)},\\
r(x)&\Let\intcc{r_1(x_1);\ldots;r_N(x_n)},\\
h(x)&\Let\intcc{h_{11}(x_1);\ldots;h_{NN}(x_N)},
\end{align*}
where $u=\intcc{u_{1};\ldots;u_{N}}$ and $x=\intcc{x_{1};\ldots;x_{N}}$ and
with the interconnection variables constrained by
$w_{ij}=y_{ji}$ for all $i,j\in\intcc{1;N},i\neq j$.
\end{definition}

The interconnection of two stochastic hybrid subsystems $\Sigma_i$
and $\Sigma_j$ from a group of $N$ subsystems is illustrated in Figure \ref{system1}.
\begin{figure}[ht]
\begin{center}
\includegraphics[width=6cm]{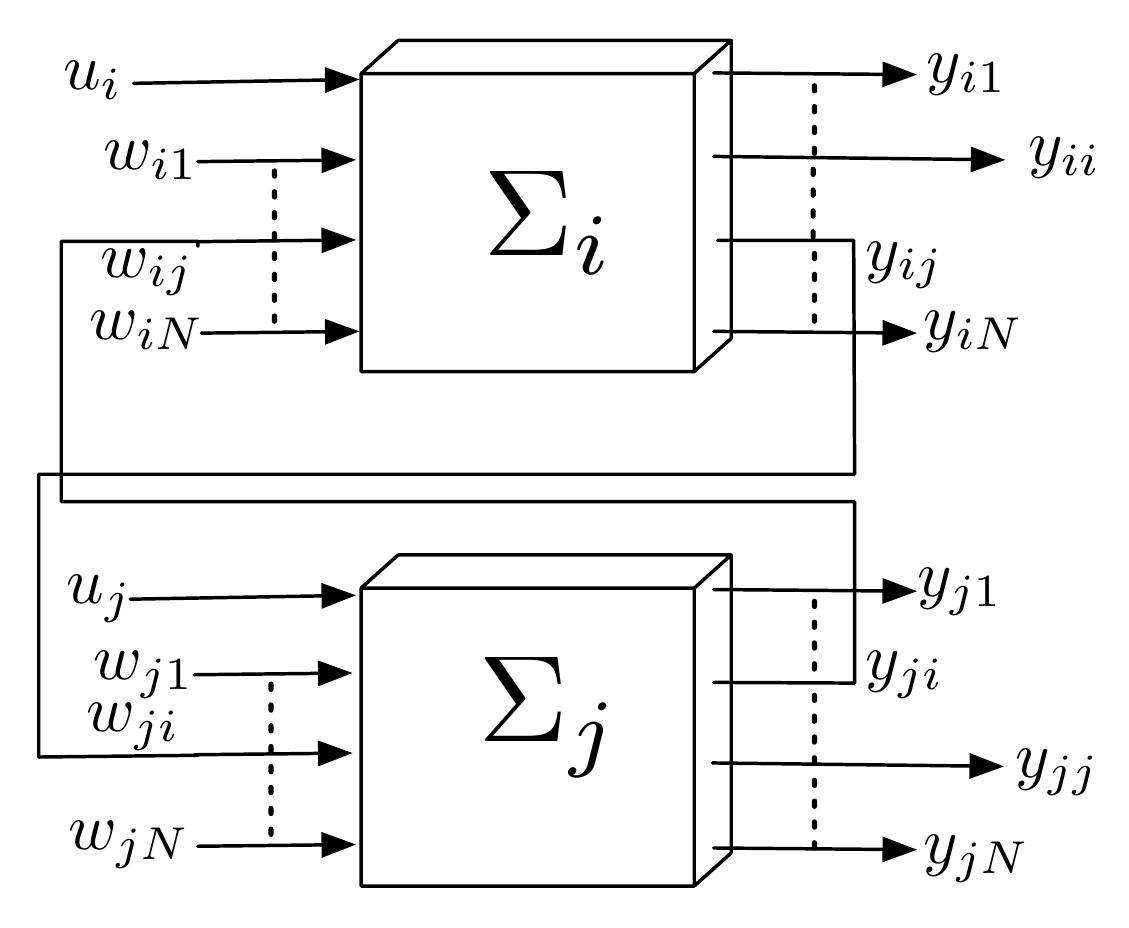}
\caption{Interconnection of two stochastic hybrid subsystems $\Sigma_i$ and $\Sigma_j$.}
\label{system1}
\end{center}
\end{figure}

\subsection{Compositional construction of abstractions and simulation functions}
We assume that we are given $N$ stochastic hybrid subsystems
$\Sigma_i=\left(\R^{n_i},\R^{m_i},\R^{p_i},\mathcal{U}_i,\mathcal{W}_i,f_i,\sigma_i,r_i,\R^{q_i},h_i\right),$
together with their corresponding abstractions $\hat\Sigma_i=(\R^{\hat n_i},\R^{\hat
m_i},\R^{p_i},\hat{\mathcal{U}}_i,\mathcal{W}_i,\hat f_i,\hat\sigma_i,\hat r_i,\R^{q_i},\hat
h_i)$ and with SSF-M$_k$ functions $V_i$ from
$\hat\Sigma_i$ to $\Sigma_i$. 
In order to provide the main compositionality result, we require the following assumption:

\begin{assumption}\label{assumption1}
For any $i,j\in\intcc{1;N}$, $i\neq j$, there exist $\mathcal{K}_\infty$ convex functions $\gamma_i$ and constants $\widetilde\lambda_i\in\R_{>0}$ and $\delta_{ij}\in\R_{\geq0}$ such that for any $s\in\R_{\geq0}$
\begin{small}
\begin{IEEEeqnarray}{l}\label{assumption}
\IEEEyesnumber
\hspace{-0.5cm}\IEEEyessubnumber\label{a:1a}  \eta_i(s)\geq\widetilde\lambda_i\gamma_i(s)\\
\hspace{-0.5cm}\IEEEyessubnumber\label{a:1b}  h_{ji}\equiv 0\implies \delta_{ij}=0 \text{ and }\\
\hspace{-0.5cm}\IEEEyessubnumber\label{a:1c}  h_{ji}\not\equiv 0\implies \rho_{i\mathrm{int}}((N-1)^{\max\{\frac{k}{2},1\}}\alpha_j^{-1}(s))\leq\delta_{ij}\gamma_j(s),
\end{IEEEeqnarray}
\end{small}
where $\eta_i$, $\alpha_i$, and $\rho_{i\mathrm{int}}$ represent the corresponding $\mathcal{K}_\infty$ functions of subsystems~$\Sigma_i$ appearing in Definition~\ref{d:sf}. 
\end{assumption}

For notational simplicity in the rest of the paper, we define matrices $\Lambda$ and $\Delta$ in $\R^{N\times{N}}$ with their components
given by
$\Lambda_{ii}=\widetilde\lambda_i$, $\Delta_{ii}=0$ for $i\in\intcc{1;N}$ and 
$\Lambda_{ij}=0$, $\Delta_{ij}=\delta_{ij}$ for
$i,j\in\intcc{1;N}, i\neq j$. Moreover, we define $\Gamma(\as):=\intcc{\gamma_1(s_1);\ldots;\gamma_N(s_N)}$, where $\as=\intcc{s_1;\ldots;s_N}$.

The next theorem provides a compositional approach on the construction of abstractions of interconnected stochastic hybrid systems and that of the corresponding SSF-M$_k$ functions.

\begin{theorem}\label{t:ic}
Consider the interconnected stochastic hybrid system
$\Sigma=\mathcal{I}(\Sigma_1,\ldots,\Sigma_N)$ induced by $N\in\N_{\geq1}$ stochastic
hybrid subsystems~$\Sigma_i$.
Suppose that each stochastic hybrid subsystem $\Sigma_i$ approximately alternatingly simulates a stochastic hybrid subsystem $\hat \Sigma_i$ with the corresponding SSF-M$_k$ function $V_i$. If Assumption \ref{assumption1} holds and there exists a vector
$\mu\in\R^N_{>0}$ such that the inequality
\begin{IEEEeqnarray}{c}\label{e:SGcondition}
\mu^T(-\Lambda+\Delta)< 0
\end{IEEEeqnarray}
is satisfied\footnote{We interpret the inequality component-wise, i.e.,
for $x\in\R^N$ we have $x<0$ iff every entry $x_i<0$,
  $i\in\{1,\ldots,N\}$.}, then 
\begin{IEEEeqnarray*}{c}
V(x,\hat x)\Let\sum_{i=1}^N\mu_iV_i(x_i,\hat x_i)
\end{IEEEeqnarray*}
is an SSF-M$_k$
function from $\hat \Sigma=\mathcal{I}(\hat
\Sigma_1,\ldots,\hat\Sigma_N)$ to $\Sigma$. 

\end{theorem}

\begin{proof}
Note that for any $x=\intcc{x_1;\ldots;x_N}$, where $x_i\in\R^{n_i}$ and $i\in[1;N]$, one obtains: $$\Vert x\Vert^k\leq\sum_{i=1}^N\Vert x_i\Vert^k,$$ for any $k\in\intcc{1,2}$ due to triangle inequality and appropriate equivalency between different norms and $$\Vert x\Vert^k=(\Vert x\Vert^2)^{\frac{k}{2}}= (\sum_{i=1}^N\Vert x_i\Vert^2)^{\frac{k}{2}}\leq N^{\frac{k}{2}-1} \sum_{i=1}^N\Vert x_i\Vert^k,$$ for any $k>2$ due to Jensen's inequality \cite{ref:Boyd} for convex functions. By combining the previous inequalities, one gets 
\begin{equation}\label{jenson}
\Vert x\Vert^k\leq N^{\max\{\frac{k}{2},1\}-1} \sum_{i=1}^N\Vert x_i\Vert^k,
\end{equation}
for any $k\geq1$ and any $x=\intcc{x_1;\ldots;x_N}$, where $x_i\in\R^{n_i}$ and $i\in[1;N]$.

First we show that inequality \eqref{e:sf:1} holds for some convex
$\mathcal{K}_\infty$ function $\alpha$. Using \eqref{jenson}, for any
$x=\intcc{x_1;\ldots;x_N}\in\R^{n}$ and 
$\hat x=\intcc{\hat x_1;\ldots;\hat x_N}\in\R^{\hat n}$, one gets:
\begin{align*}
\Vert \hat h(\hat x)-h(x) \Vert^k
&\le
N^{\max\{\frac{k}{2},1\}-1}\sum_{i=1}^N \Vert \hat h_{ii}(\hat x_i)-h_{ii}(x_i) \Vert^k\le 
N^{\max\{\frac{k}{2},1\}-1}\sum_{i=1}^N \Vert \hat h_{i}(\hat x_i)-h_{i}(x_i) \Vert^k\\ 
&\le
N^{\max\{\frac{k}{2},1\}-1}\sum_{i=1}^N \alpha_i^{-1}(V_i( x_i, \hat x_i))\le \ol\alpha(V(x,\hat x)),
\end{align*}where
$\overline\alpha$ is a $\mathcal{K}_\infty$ function defined as 
\begin{align}\notag
	\ol\alpha(s) &\Let \left \{ \begin{array}{cc} \max\limits_{\as {\ge 0}} & N^{\max\{\frac{k}{2},1\}-1}\sum_{i=1}^N \alpha_i^{-1}(s_i) \\ \st & \mu^T \as=s. \end{array} \right. 
\end{align}
Now we show that $\ol\alpha$ is a concave function. Let us recall that by assumptions $\alpha_i$ are convex functions and, hence, $\alpha_i^{-1}$ are concave\footnote{\label{inverse}Note that the inverse of a strictly increasing convex (resp. concave) function is a strictly
increasing concave (resp. convex) one.}. Thus, from an optimization point of view, the function $\ol\alpha$ is a \emph{perturbation} function which is known to be a concave function; see \cite[Section 5.6.1, p.\ 249]{ref:Boyd} for further details. By defining the convex\footnotemark[\getrefnumber{inverse}]
$\mathcal{K}_\infty$ function $\alpha(s)=\overline\alpha^{-1}(s)$, $\forall s\in\R_{\ge0}$, one obtains
$$\alpha(\Vert \hat h(\hat x)-h(x)\Vert^k)\le V( x, \hat x),$$satisfying inequality \eqref{e:sf:1}.
Now we show that inequality \eqref{inequality1} holds as well.
Consider any 
$x=\intcc{x_1;\ldots;x_N}\in\R^{n}$,
$\hat x=\intcc{\hat x_1;\ldots;\hat x_N}\in\R^{\hat n}$, and
$\hat u=\intcc{\hat u_{1};\ldots;\hat u_{N}}\in\R^{\hat m}$. For any $i\in[1;N]$, there exists $u_i\in\R^{m_i}$, consequently, a vector $u=\intcc{u_{1};\ldots;u_{N}}\in\R^{m}$, satisfying~\eqref{inequality1} for each pair of subsystems $\Sigma_i$ and $\hat\Sigma_i$
with the internal inputs given by $w_{ij}=h_{ji}(x_j)$ and $\hat
w_{ij}=\hat h_{ji}(\hat x_j)$.
We derive the chain of inequalities in \eqref{chain}, where we
use the inequalities \eqref{jenson} and: 
$$
\rho_{i\mathrm{int}}(r_1+\cdots+r_{N-1})\leq\sum_{i=1}^{N-1}\rho_{i\mathrm{int}}((N-1)r_i),
$$
which are valid for any $k\geq1$, $\rho_{i\mathrm{int}}\in\mathcal{K}_\infty$, $x_i\in \R^{n_i}$, and
any $r_i\in\R_{\ge0}$, \mbox{$i\in\intco{1;N}$}. Note that if
$\rho_{i\mathrm{int}}$ satisfies the triangle inequality, one gets the less conservative inequality
$$
\rho_{i\mathrm{int}}(r_1+\cdots+r_{N-1})\leq\sum_{i=1}^{N-1}\rho_{i\mathrm{int}}(r_i),
$$
and it suffices that~\eqref{assumption2} holds instead
of~\eqref{a:1c}.
\begin{figure*}[t]
\rule{\textwidth}{0.1pt}
\begin{align}\notag
\mathcal{L} V\left(x, \hat{x}\right)
&=
\sum_{i=1}^N\mu_i\mathcal{L} V_i\left( x_i,
\hat{x}_i\right)\leq\sum_{i=1}^N\mu_i\left(-\eta_i(V_i( x_i,\hat x_i))+\rho_{i\mathrm{int}}(\left\Vert w_{i}-\hat w_{i}\right\Vert^k)+\rho_{i\mathrm{ext}}(\Vert \hat u_i\Vert^k)\right)\\\notag
&
\leq
\sum_{i=1}^N\mu_i\left(-\eta_i(V_i( x_i,\hat x_i))+\sum_{j=1,j\neq i}^N\rho_{i\mathrm{int}}((N-1)^{\max\{\frac{k}{2},1\}}\left\Vert w_{ij}-\hat w_{ij}\right\Vert^k)+\rho_{i\mathrm{ext}}(\Vert \hat u_i\Vert^k)\right)\\\notag
&
\le
\sum_{i=1}^N\mu_i\left(-\eta_i(V_i( x_i,\hat x_i))+\sum_{j=1,j\neq i}^N\rho_{i\mathrm{int}}((N-1)^{\max\{\frac{k}{2},1\}}\left\Vert y_{ji}-\hat y_{ji}\right\Vert^k)+\rho_{i\mathrm{ext}}(\Vert \hat u_i\Vert^k)\right)\\\notag
&\le
\sum_{i=1}^N\mu_i\left(-\eta_i(V_i( x_i,\hat x_i))+\sum_{j=1,j\neq i}^N\rho_{i\mathrm{int}}((N-1)^{\max\{\frac{k}{2},1\}}\Vert h_{j}(x_j)-\hat h_{j}(\hat x_j)\Vert^k)+\rho_{i\mathrm{ext}}(\Vert \hat u_i\Vert^k)\right)\\\notag
&
\le
\sum_{i=1}^N\mu_i\left(-\eta_i(V_i( x_i,\hat x_i))+\sum_{j=1,j\neq i}^N\rho_{i\mathrm{int}}\left((N-1)^{\max\{\frac{k}{2},1\}}\alpha_j^{-1}\left(V_j( x_j,\hat x_j)\right)\right)+\rho_{i\mathrm{ext}}(\Vert \hat u_i\Vert^k)\right)\\\notag
&
\leq
\sum_{i=1}^N\mu_i\left(-\widetilde\lambda_i\gamma_i(V_i( x_i,\hat x_i))+\sum_{i\neq{j},j=1}^N\delta_{ij}\gamma_j\left(V_j\left( x_j,\hat{x}_j\right)\right)+\rho_{i\mathrm{ext}}(\left\Vert \hat{u}_i\right\Vert^k)\right)\\\label{chain}
&
=
\mu^\top\left(-\Lambda+\Delta\right)\Gamma\left([V_1\left( x_1,\hat{x}_1\right);\ldots;V_N\left( x_N,\hat{x}_N\right)]\right)+\sum_{i=1}^N\mu_i\rho_{i\mathrm{ext}}(\left\Vert \hat{u}_i\right\Vert^k).
\end{align}
\rule{\textwidth}{0.1pt}
\end{figure*}
Define the functions 
\begin{subequations}
\begin{align}
	\label{lambda}
	\eta(s) &\Let \left \{ \begin{array}{cc} \min\limits_{\as  {\ge 0}} & -\mu^T\left(-\Lambda+\Delta\right)\Gamma(\as) \\ \st & \mu^T \as=s, \end{array} \right. \\
	\label{rho}
	\rho_{\mathrm{ext}}(s) &\Let \left \{ \begin{array}{cc} \max\limits_{\as {\ge 0}} & \sum_{i=1}^N\mu_i\rho_{i\mathrm{ext}}(s_i) \\ \st & \|\as \| \le s. \end{array} \right. 
\end{align}
\end{subequations}
By construction, we readily have
\begin{align}\notag
\dot V\left( x,\hat{x}\right)\leq-\eta\left(V\left( x,\hat{x}\right)\right)+\rho_{\mathrm{ext}}(\left\Vert \hat{u}\right\Vert^k),
\end{align}
where the functions $\eta$ and $\rho_{\mathrm{ext}}$ are $\mathcal{K}_\infty$ functions. It remains to show that $\eta$ is a convex function and $\rho_{\mathrm{ext}}$ is a concave one. Let us recall that by assumptions $\mu^T\left(-\Lambda+\Delta\right) < 0$ and $\gamma_i$, the $i$-th element of $\Gamma$, is convex. Thus, the function $\eta$ in \eqref{lambda} is a perturbation function which is a convex one. Note that by assumption each function $\rho_{i\mathrm{ext}}$ is concave, and for the same reason as above, the function \eqref{rho} is also concave. Hence, we conclude that $V$ is an SSF-M$_k$ function from $\hat \Sigma$ to $\Sigma$. 
\end{proof}

\begin{remark}\label{r:diw11}
As shown in~\cite[Lemma~3.1]{DIW11}, a vector
$\mu\in\R^N_{>0}$ satisfying $\mu^T(-\Lambda+\Delta)< 0$ exists if and only if the spectral radius of
$\Lambda^{-1}\Delta$ is strictly less than one.
\end{remark}

\begin{remark}\label{r:triangle}
If the functions $\rho_{i\mathrm{int}}$, $i\in[1;N]$, satisfy the triangle
inequality, 
$\rho_{i\mathrm{int}}(a+b)\leq\rho_{i\mathrm{int}}(a)+\rho_{i\mathrm{int}}(b)$
for all non-negative values of $a$ and $b$, then the condition~\eqref{a:1c} reduces to
\begin{align}\label{assumption2}
h_{ji}\not\equiv 0\implies\rho_{i\mathrm{int}}((N-1)^{\max\{\frac{k}{2},1\}-1}\alpha_j^{-1}(s))\leq\delta_{ij}\gamma_j(s).
\end{align}
\end{remark}

Figure \ref{composition1} illustrates schematically the result of
Theorem~\ref{t:ic}. 

\begin{figure}[t]
\begin{center}
\includegraphics[width=10cm]{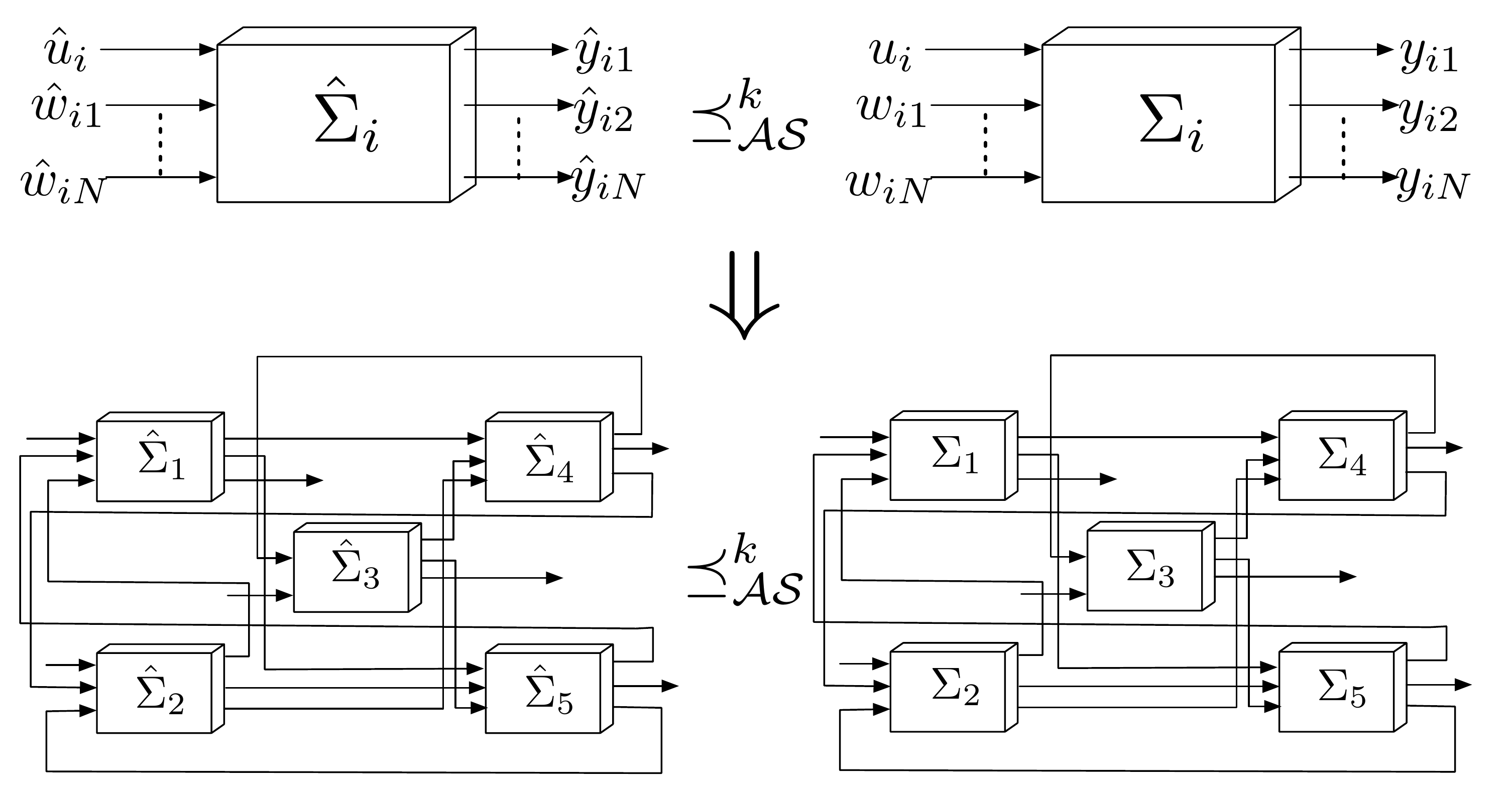}
\caption{Compositionality results.}
\label{composition1}
\end{center}
\end{figure}

\section{Jump Linear Stochastic Systems}
\label{s:lin}

In this section, we 
focus on a specific class of stochastic hybrid systems, namely, jump linear stochastic systems (JLSS) \cite{julius1} and
\emph{quadratic} SSF-M$_2$ functions $V$. 
In the first part, we assume that we are given an abstraction $\hat
\Sigma$ and provide conditions under which $V$ is an SSF-M$_2$
function. In the second part we show how to construct the
abstraction $\hat \Sigma$ together with the SSF-M$_2$ function $V$.

A JLSS is defined as a stochastic hybrid system with 
the drift, diffusion, reset, and output functions given by
\begin{align}\notag
\hspace{-10mm}\diff \xi(t)=&(A\xi(t)+B\nu(t)+D\omega(t))\diff t+E\xi(t)\diff W_t+\sum_{i=1}^{\widetilde{q}}R_i\xi(t)\diff P^i_t,\\\label{e:lin:sys}
\zeta(t)=&C\xi(t),
\end{align}where
\begin{IEEEeqnarray}{c,c,c,c,c,c,c}\nonumber
A\in\R^{n\times n},&
B\in\R^{n\times m},&
D\in\R^{n\times p},& 
E\in\R^{n\times n},&
R_i\in\R^{n\times n},&\forall i\in[1;\widetilde{q}],&
C\in\R^{q\times n}.
\end{IEEEeqnarray}
The matrices $R_i$, $\forall i\in[1;\widetilde{q}]$,
parametrize the jump associated with event $i$. 
We use the tuple
\begin{IEEEeqnarray*}{c}
\Sigma=(A,B,C,D,E,\mathsf{R}),
\end{IEEEeqnarray*}
where $\mathsf{R}=\left\{R_1,\ldots,R_{\widetilde{q}}\right\}$, to refer to a JLSS of
the form~\eqref{e:lin:sys}. Note that in this section we consider JLSS driven by a scalar Brownian motion for the sake of simple presentation, though the proposed results can be readily generalized for the systems driven by multi-dimensional Brownian motions as well.

\subsection{Quadratic SSF-M$_2$ functions}

In this section, we assume that for some constant
$\widehat\kappa\in\R_{>0}$ there exist a positive definite matrix $M\in\R^{n\times n}$ and matrix
$K\in\R^{m\times n}$ such that the matrix inequalities
\begin{align}\label{e:lin:con1}
C^TC&\preceq M, \\\label{e:lin:con11}
\left(A+BK+\sum_{i=1}^{\widetilde q}\lambda_iR_i\right)^TM+M\left(A+BK+\sum_{i=1}^{\widetilde q}\lambda_iR_i\right)+E^TME+\sum_{i=1}^{\widetilde q}\lambda_iR_i^TMR_i&\preceq -\widehat\kappa M,
\end{align}
hold. 

Note that condition \eqref{e:lin:con11} is sufficient and necessary for the asymptotic stability of $\Sigma=(A,B,C,0_{n\times p},E,\mathsf{R})$ equipped with a linear feedback control law $u=Kx$ in the mean square sense (second moment)\footnote{A stochastic hybrid system $\Sigma$ is said to be asymptotically stable in the mean square sense if all $\sigalg_0$-measurable initial states $a$ yield $\lim_{t\rightarrow\infty}\EE[\Vert\xi(t)\Vert^2]=0$.} as showed in the next lemma. Condition \eqref{e:lin:con1} is always satisfied for any positive definite matrix $M$ up to multiplication by some positive scalar which does not violate the satisfaction of \eqref{e:lin:con11}.

\begin{lemma}\label{lemma111}
A JLSS $\Sigma=(A,B,C,0_{n\times p},E,\mathsf{R})$ equipped with a linear feedback control law $u=Kx$ is asymptotically stable in the mean square sense if and only if there exists a positive definite matrix $M\in\R^{n\times n}$ such that the matrix inequality \eqref{e:lin:con11} is satisfied for given feedback gain $K$ and some positive constant $\widehat\kappa$.
\end{lemma}
The proof of Lemma \ref{lemma111} is provided in the Appendix.

The matrices $K$ and $M$ in \eqref{e:lin:con1} and \eqref{e:lin:con11} can be computed jointly
using semidefinite programming as explained in the following lemma.

\begin{lemma}\label{l:lmi}
Denoting $\ol K=KM^{-1}$ and $\ol M=M^{-1}$, matrix inequalities \eqref{e:lin:con1} and \eqref{e:lin:con11} are equivalent to the following linear matrix inequalities:
\begin{align}\label{e:lin:con1hat}
\begin{bmatrix}\ol M & \ol MC^T \\ C\ol M & I_q\end{bmatrix}\succeq0\\\label{e:lin:con11hat}
\begin{bmatrix}\ol M & 0 & \cdots & 0 & E\ol M \\ 0 & \ol M & \ddots & \vdots & \lambda_{\widetilde q}^{\frac{1}{2}}R_{\widetilde q}\ol M \\ \vdots & \ddots & \ddots & 0 & \vdots \\ 0 & \cdots & 0 & \ol M & \lambda_1^{\frac{1}{2}}R_1\ol M \\ \ol ME^T & \lambda_{\widetilde q}^{\frac{1}{2}}\ol MR_{\widetilde q}^T & \cdots & \lambda_1^{\frac{1}{2}}\ol MR_1^T & \ol Q \end{bmatrix}\succeq0,
\end{align}
where $0$'s denote zero matrices of appropriate dimensions and 
\begin{align*}
\ol Q\Let&-\widehat\kappa\ol M-\ol M\Big(A+\sum_{i=1}^{\widetilde q}\lambda_iR_i\Big)^T-\Big(A+\sum_{i=1}^{\widetilde q}\lambda_iR_i\Big)\ol M-\ol K^TB^T-B\ol K.
\end{align*}
\end{lemma}

The proof is a simple consequence of using Schur complements \cite{ref:Boyd} and is omitted here for the sake of brevity.

Here, we consider a quadratic SSF-M$_2$ function of the following form 
\begin{align}\label{e:lin:sf}
V(x,\hat x)=(x-P\hat x)^TM(x-P\hat x),
\end{align}
where $P$ is a matrix of appropriate
dimension. 
Assume
that the equalities 
\begin{IEEEeqnarray}{rCl}\label{e:lin:con2} \IEEEyesnumber
\IEEEyessubnumber\label{e:lin:con2a} AP&=&P\hat A-BQ\\
\IEEEyessubnumber\label{e:lin:con2b} D&=&P\hat D-BS\\
\IEEEyessubnumber\label{e:lin:con2c} CP&=&\hat C\\
\IEEEyessubnumber\label{e:lin:con2g}  EP&=&P\hat E\\
\IEEEyessubnumber\label{e:lin:con2h}  R_iP&=&P \hat R_i,~~\forall i\in[1;\widetilde q],
\end{IEEEeqnarray}
hold for some matrices $Q$ and $S$ of appropriate dimensions. In the following theorem, we show that those conditions imply that \eqref{e:lin:sf} is an SSF-M$_2$ function from $\hat \Sigma$ to $\Sigma$. 

\begin{theorem}\label{t:lin:suf}
Consider two JLSS $\Sigma=(A,B,C,D,E,\mathsf{R})$
and $\hat \Sigma=(\hat A,\hat B,\hat C,\hat
D,\hat E,\hat {\mathsf{R}})$ with $p=\hat p$ and  $q=\hat q$. Suppose that there exist matrices 
$M$, $K$, $P$, $Q$, and $S$ satisfying~\eqref{e:lin:con1},~\eqref{e:lin:con11}, and~\eqref{e:lin:con2}, for some constant $\widehat\kappa\in\R_{>0}$. Then, $V$ defined in~\eqref{e:lin:sf} is an SSF-M$_2$ function from~$\hat\Sigma$ to~$\Sigma$.
\end{theorem}

\begin{proof}
Note that $V$ is twice continuously differentiable.
We show that for every
$x\in\R^n$, $\hat x\in\R^{\hat n}$, $\hat u\in\R^{\hat m}$, $\hat w\in\R^p$, there exists $u\in\R^m$ such that for all $w\in\R^p$, $V$ satisfies $\Vert Cx-\hat C\hat x\Vert^2\le V(x,\hat
x)$ and
\begin{align}\notag
\mathcal{L} V(x,\hat x)
\Let&
\frac{\partial V(x,\hat x)}{\partial x}
(Ax+Bu+Dw)+
\frac{\partial V(x,\hat x)}{\partial \hat x}
(\hat A\hat x+\hat B\hat u+\hat D\hat w)\\\notag&+\frac{1}{2} \text{Tr} \left(\begin{bmatrix} Ex \\ \hat E\hat{x} \end{bmatrix}\left[x^TE^T~~\hat x^T\hat E^T\right] \begin{bmatrix}
\partial_{x,x} V & \partial_{x,\hat{x}} V \\ \partial_{\hat{x},x} V & \partial_{\hat{x},\hat{x}} V
\end{bmatrix}	\right)+\sum_{i=1}^{\widetilde q}\lambda_i(V(x+R_ix,\hat{x}+\hat R_i\hat{x})-V(x,\hat{x}))\\\label{e:t:lin:main:1}
\le&-(\widehat\kappa-\pi) V(x,\hat x)+\frac{2\Vert\sqrt{M}D\Vert^2}{\pi}\Vert w-\hat w\Vert^2+\frac{2\Vert\sqrt{M}(B\widetilde R-P\hat B)\Vert^2}{\pi}\Vert\hat u\Vert^2,
\end{align}
for any positive constant $\pi<\widehat\kappa$ and some matrix $\widetilde R$ of appropriate dimension.

From~\eqref{e:lin:con2c}, we have $\Vert Cx-\hat C\hat x\Vert^2=(x-P\hat
x)^TC^TC(x-P\hat x)$ and using $M\succeq C^TC$, it can be readily verified that 
$\Vert Cx-\hat C\hat x\Vert^2\le V(x,\hat x)$ holds for all $x\in\R^n$, $\hat
x\in\R^{\hat n}$. We proceed with showing the inequality in~\eqref{e:t:lin:main:1}. 
Note that 
\begin{align*}
&\partial_x V(x,\hat x)
= 
2(x-P\hat x)^TM,~ 
\partial_{\hat x} V(x,\hat x)
=
-2(x-P\hat x)^TMP,~\partial_{x,x} V(x,\hat x)
= 2M,~\text{and}\\&\partial_{\hat x,\hat x} V(x,\hat x) =P^T \partial_{x,x} V(x,\hat x)  P,~
\partial_{x,\hat x} V(x,\hat x)= \left(\partial_{\hat x, x} V(x,\hat x)\right)^T=-\partial_{x,x} V(x,\hat x) P
\end{align*}
holds. Given any $x\in\R^n$, $\hat x\in\R^{\hat n}$, $\hat u\in\R^{\hat m}$, and $\hat w\in\R^p$, we choose $u\in\R^{m}$ via the following linear \emph{interface} function:
\begin{align}\label{e:lin:int}
u=\nu_{\hat \nu}(x,\hat x,\hat u,\hat w):=K(x-P\hat x)+Q\hat x+\widetilde R\hat u+S \hat w,
\end{align}
for some matrix $\widetilde R$ of appropriate dimension.

By using the equations~\eqref{e:lin:con2a}
and~\eqref{e:lin:con2b} and the definition of the interface function in~\eqref{e:lin:int}, we simplify
\begin{IEEEeqnarray*}{l}
Ax+B\nu_{\hat \nu}(x,\hat x, \hat u,\hat w)+Dw-P(\hat A\hat x+\hat
B\hat u+\hat D\hat w)
\end{IEEEeqnarray*}
to $(A+BK)(x-P\hat x)+D(w-\hat w)+(B\widetilde R-P\hat B)\hat u$
and obtain the following expression for $\mathcal{L} V(x,\hat x)$:
\begin{IEEEeqnarray*}{l}
\mathcal{L} V(x,\hat x)=2(x-P\hat x)^TM
\big[
(A+BK)(x-P\hat x)
+
D(w-\hat w)
+
(B\widetilde R-P\hat B)\hat u
\big]\\
+\begin{bmatrix}x \\ \hat x\end{bmatrix}^T\begin{bmatrix}E^T & 0 \\ 0 & \hat E^T\end{bmatrix}\begin{bmatrix}M & -MP \\ -P^TM & P^TMP\end{bmatrix}\begin{bmatrix}E & 0 \\ 0 & \hat E\end{bmatrix}\begin{bmatrix}x \\ \hat x\end{bmatrix}+\begin{bmatrix}x \\ \hat x\end{bmatrix}^T\sum_{i=1}^{\widetilde q}\lambda_i\begin{bmatrix}R_i^T & 0 \\ 0 & \hat R_i^T\end{bmatrix}\begin{bmatrix}M & -MP \\ -P^TM & P^TMP\end{bmatrix}\begin{bmatrix}x \\ \hat x\end{bmatrix}\\
+\begin{bmatrix}x \\ \hat x\end{bmatrix}^T\begin{bmatrix}M & -MP \\ -P^TM & P^TMP\end{bmatrix}\sum_{i=1}^{\widetilde q}\lambda_i\begin{bmatrix}R_i & 0 \\ 0 & \hat R_i\end{bmatrix}\begin{bmatrix}x \\ \hat x\end{bmatrix}\\
+\begin{bmatrix}x \\ \hat x\end{bmatrix}^T\sum_{i=1}^{\widetilde q}\lambda_i\begin{bmatrix}R_i^T & 0 \\ 0 & \hat R_i^T\end{bmatrix}\begin{bmatrix}M & -MP \\ -P^TM & P^TMP\end{bmatrix}\begin{bmatrix}R_i & 0 \\ 0 & \hat R_i\end{bmatrix}\begin{bmatrix}x \\ \hat x\end{bmatrix},
\end{IEEEeqnarray*}
where $0$'s denote zero matrices of appropriate dimensions. We use~\eqref{e:lin:con2g} and~\eqref{e:lin:con2h} to obtain the following expression for $\mathcal{L} V(x,\hat x)$:
\begin{align*}
\mathcal{L} V(x,\hat x)=& (x-P\hat x)^T\Big[\Big(A+BK+\sum_{i=1}^{\widetilde q}\lambda_iR_i\Big)^TM+M\Big(A+BK+\sum_{i=1}^{\widetilde q}\lambda_iR_i\Big)\\&+E^TME+\sum_{i=1}^{\widetilde q}\lambda_iR_i^TMR_i\Big](x-P\hat x)+2(x-P\hat x)^TM
\big[
D(w-\hat w)
+
(B\widetilde R-P\hat B)\hat u
\big].
\end{align*}
Using Young's inequality \cite{Young225} as $$ab\leq \frac{\epsilon}{2}a^2+\frac{1}{2\epsilon}b^2,$$ for any $a,b\geq0$ and any $\epsilon>0$,
and with the help of Cauchy-Schwarz inequality and \eqref{e:lin:con11} one gets the following upper bound for $\mathcal{L} V(x,\hat x)$:
\begin{align*}
&\mathcal{L} V(x,\hat x)\le -\widehat\kappa V(x,\hat x)+\pi V(x,\hat x)+\frac{2\Vert\sqrt{M}D\Vert^2}{\pi}\Vert w-\hat w\Vert^2+\frac{2\Vert\sqrt{M}(B\widetilde R-P\hat B)\Vert^2}{\pi}\Vert\hat u\Vert^2,
\end{align*}
for any positive constant $\pi<\widehat\kappa$.

Using this computed upper bound, we obtain~\eqref{e:t:lin:main:1} which completes the proof. Note that the $\mathcal{K}_\infty$ functions
$\alpha$, $\eta$, $\rho_{\mathrm{ext}}$, and $\rho_{\mathrm{int}}$,
in Definition~\ref{d:sf} associated with the SSF-M$_2$ function
in \eqref{e:lin:sf}
are given by $\alpha(s):=s$, $\eta(s):=(\widehat\kappa-\pi) s$, $\rho_{\mathrm{ext}}(s):=\frac{2\Vert\sqrt{M}(B\widetilde R-P\hat B)\Vert^2}{\pi} s$ and
$\rho_{\mathrm{int}}(s):=\frac{2\Vert\sqrt{M}D\Vert^2}{\pi} s$, $\forall s\in\R_{\ge0}$.

\end{proof}

\begin{remark}\label{bound}
Using the linear functions $\alpha$ and $\eta$, as computed in Theorem \ref{t:lin:suf}, the functions $\beta$, $\gamma_{\mathrm{ext}}$, and $\gamma_{\mathrm{int}}$, appearing in Theorem \ref{theorem1}, are simplified as the following: $\beta(r,t)\Let r\mathsf{e}^{-(\widehat\kappa-\pi) t}$, $\gamma_{\mathrm{ext}}(r)\Let \frac{1}{\widehat\kappa-\pi}\rho_{\mathrm{ext}}(r)$, and $\gamma_{\mathrm{int}}(r)\Let \frac{1}{\widehat\kappa-\pi}\rho_{\mathrm{int}}(r)$ for any $r,t\in\R_{\geq0}$.
\end{remark}

\begin{remark}
Note that Theorem~\ref{t:lin:suf} does not impose any
condition on matrix $\widetilde R$. Similar to the results in~\cite[Proposition 1]{GP09} for the deterministic case, we propose a
choice of $\widetilde R$ which minimize function $\rho_{\mathrm{ext}}$ for $V$. The choice of $\widetilde R$ minimizing $\rho_{\mathrm{ext}}$ is given by
\begin{IEEEeqnarray}{c}\label{e:lin:con:tildeR}
\widetilde R=(B^TMB)^{-1} B^TM P\hat B.
\end{IEEEeqnarray}
\end{remark}

\begin{remark}\label{r:cutConnection}
Consider $\Sigma_i=(A_i,B_i,C_i,D_i,E_i,\mathsf{R}_i)$ and its abstraction $\hat\Sigma_i=(\hat A_i,\hat B_i,\hat C_i,\hat D_i,\hat E_i,\hat{\mathsf{R}}_i)$. Assume $D_i=\intcc{d_i^1\cdots d_i^p}$ and $\hat D_i=\intcc{\hat d_i^1 \cdots \hat d_i^p}$. Using equation \eqref{e:lin:con2b}, one can readily conclude that if $d_i^j\in \im B$, for some $j\in\intcc{1;p}$, then the corresponding $\hat d_i^j$ can be chosen as $\hat d_i^j=0_{\hat n}$. This choice for columns of $\hat D$ makes the interconnection topology of abstract subsystems potentially simpler and, hence, their analysis easier. We refer the interested readers to Section \ref{case_study} for an example of such choice for $\hat D$.
\end{remark}

As of now, we derived various conditions on the
original system $\Sigma$, the abstraction $\hat \Sigma$, and the matrices appearing in~\eqref{e:lin:sf} and~\eqref{e:lin:int}, to
ensure that~\eqref{e:lin:sf} is
an SSF-M$_2$ function from $\hat\Sigma$ to $\Sigma$ with the corresponding interface function in \eqref{e:lin:int} lifting any control policy designed for $\hat\Sigma$ to the one for $\Sigma$.
However, those conditions do not impose any requirements on the abstract external
input matrix $\hat B$. As an example, one can choose $\hat B=I_{\hat n}$ which makes the abstract system $\hat \Sigma$ fully actuated and, hence, the synthesis problem over $\hat \Sigma$ much easier.
Similar to~\cite[Subsection 4.1]{GP09} in the context of deterministic control systems, one can also choose an external input matrix $\hat B$ which
\emph{preserves} all the behaviors of the original JLSS $\Sigma$ on the
abstraction $\hat \Sigma$: for every trajectory
$(\xi,\zeta,\nu,\omega)$ of $\Sigma$ there exists a trajectory
$(\hat \xi,\hat \zeta,\hat \nu,\hat \omega)$ of $\hat \Sigma$ such
that $\hat\zeta=\zeta$ $\PP$-a.s..

Note that using the following choice of external input matrix $\hat B$, the results in \cite{majid17} for the linear deterministic control system are fully recovered by the corresponding ones here providing that the JLSS is not affected by any noise, implying that $E$, $\hat E$, $R_i$, and $\hat R_i$, $\forall i\in[1;\widetilde q]$, are identically zero.

\begin{theorem}\label{t:B}
Consider two JLSS $\Sigma=(A,B,C,D,E,\mathsf{R})$
and $\hat \Sigma=(\hat A,\hat B,\hat C,\hat
D,\hat E,\hat {\mathsf{R}})$ with $p=\hat p$ and  $q=\hat q$. Suppose that there
exist matrices $P$, $Q$, and $S$ satisfying~\eqref{e:lin:con2} and that the abstract external input matrix $\hat B$ is
given by
\begin{IEEEeqnarray}{c}\label{e:lin:con:B}
\hat B=[\hat PB\;\; \hat PAG],
\end{IEEEeqnarray}
where $\hat P$ and $G$ are assumed to satisfy 
\begin{IEEEeqnarray}{rCl}\label{e:lin:con3} \IEEEyesnumber
\IEEEyessubnumber\label{e:lin:con3a} C&=&\hat C\hat P\\
\IEEEyessubnumber\label{e:lin:con3b} I_n&=&P\hat P+GF\\
\IEEEyessubnumber\label{e:lin:con3c} I_{\hat n}&=& \hat PP\\
\IEEEyessubnumber\label{e:lin:con3d} 0_{\hat n\times n}&=&\hat PEGF\\
\IEEEyessubnumber\label{e:lin:con3e}  0_{\hat n\times n}&=&\hat PR_iGF,~~\forall i\in[1;\widetilde q],
\end{IEEEeqnarray}
for some matrix $F$. Then, for every trajectory
  $(\xi,\zeta,\nu,\omega)$~of~$\Sigma$ there exists a trajectory
  $(\hat \xi,\hat \zeta,\hat \nu,\hat \omega)$ of $\hat \Sigma$ so
  that $\zeta=\hat \zeta$ holds $\PP$-a.s..
\end{theorem}

\begin{proof}
Let  $(\xi,\zeta,\nu,\omega)$ be a trajectory of~$\Sigma$. We are
going to show that $(\hat \xi,\hat \zeta,\hat \nu,\omega)$ with 
\begin{IEEEeqnarray*}{c;c;t;c}
\hat \zeta=\zeta,~&
\hat \xi= \hat P\xi,~
& and &
\hat \nu
=
\begin{bmatrix}
\nu-Q\hat P\xi-S\omega\\
F\xi
\end{bmatrix},
\end{IEEEeqnarray*}
$\PP$-a.s. is a trajectory of $\hat \Sigma$. We use \eqref{e:lin:con2g}, \eqref{e:lin:con2h}, \eqref{e:lin:con3b}, \eqref{e:lin:con3c}, \eqref{e:lin:con3d}, and \eqref{e:lin:con3e} and derive
\begin{align*}
\diff\hat P \xi&=(\hat PA\xi+\hat PB\nu+\hat PD\omega)\diff t+\hat PE\xi\diff W_t+\sum_{i=1}^{\widetilde q}\hat PR_i\xi\diff P_t^i\\
&=
(\hat PAP\hat P\xi+\hat PA(I_n-P\hat P)\xi+\hat PB\nu+\hat PD\omega)\diff t+\hat PE(P\hat P+GF)\xi\diff W_t+\sum_{i=1}^{\widetilde q}\hat PR_i(P\hat P+GF)\xi\diff P_t^i\\
&=
(\hat PAP\hat P\xi+\hat PAGF\xi+\hat PB\nu+\hat PD\omega)\diff t+\hat PP\hat E\hat P\xi\diff W_t+\sum_{i=1}^{\widetilde q}\hat PP\hat R_i\hat P\xi\diff P_t^i\\
&=
(\hat PAP\hat P\xi+\hat PAGF\xi+\hat PB\nu+\hat PD\omega)\diff t+\hat E\hat P \xi\diff W_t+\sum_{i=1}^{\widetilde q}\hat R_i\hat P \xi\diff P_t^i.
\end{align*}
Now we use the equations~\eqref{e:lin:con2a} and \eqref{e:lin:con2b} and the definition of $\hat B$ and $\hat \nu$ to derive
\begin{align*}
\diff\hat P \xi
=&
\big(\hat P(P \hat A-BQ)\hat P\xi+\hat PAGF\xi+\hat PB\nu+\hat P(P\hat
D-BS)\omega\big)\diff t+\hat E\hat P \xi\diff W_t+\sum_{i=1}^{\widetilde q}\hat R_i\hat P \xi\diff P_t^i\\
=&
(\hat A\hat P \xi+[\hat PB\; \hat PAG]\hat\nu+\hat D\omega)\diff t+\hat E\hat P \xi\diff W_t+\sum_{i=1}^{\widetilde q}\hat R_i\hat P \xi\diff P_t^i\\
=&
(\hat A\hat P \xi+\hat B\hat\nu+\hat D\omega)\diff t+\hat E\hat P \xi\diff W_t+\sum_{i=1}^{\widetilde q}\hat R_i\hat P \xi\diff P_t^i
\end{align*}
showing that $(\hat P \xi,\hat \zeta,\hat \nu,\omega)$ is a trajectory of
$\hat \Sigma$. From $C=\hat C\hat P$ in \eqref{e:lin:con3a}, it follows that $\hat \zeta=\zeta$ $\PP$-a.s.
which concludes the proof.
\end{proof}

\subsection{Construction of abstractions}
In this subsection, we provide constructive methods to compute the
abstraction $\hat \Sigma$ along with the various matrices involved in
the definition of the stochastic simulation function and its corresponding interface function. 

First, let us recall Lemma 2 in~\cite{GP09}, showing that there
exist matrices $\hat A$ and $Q$ satisfying~\eqref{e:lin:con2a} if and only if columns of $P$
span an $(A,B)$-controlled invariant subspace, see e.g. \cite[Definition~4.1.1]{BM92}.

\begin{lemma}\label{l:lin:con:AQ}
Consider matrices $A$, $B$, and $P$. There exist matrices $\hat A$ and
$Q$ satisfying~\eqref{e:lin:con2a} if and only if
\begin{IEEEeqnarray}{c}\label{e:lin:con:AQ}
\im AP\subseteq \im P+\im B.
\end{IEEEeqnarray}
\end{lemma}

Given that $P$ satisfies~\eqref{e:lin:con:AQ}, it is straightforward to
compute $\hat A$ and $Q$ such that~\eqref{e:lin:con2a} holds, by solving $\hat n$ linear equations.

Similar to  Lemma~\ref{l:lin:con:AQ}, we give necessary and sufficient conditions for the
existence of matrices $\hat D$ and $S$ appearing in
condition~\eqref{e:lin:con2b}.
\begin{lemma}\label{l:lin:con:DS}
Given $P$ and $B$, there exist matrices $\hat D$ and
$S$ satisfying~\eqref{e:lin:con2b} if and only if
\begin{IEEEeqnarray}{c}\label{e:lin:con:DS}
\im D\subseteq \im P+\im B.
\end{IEEEeqnarray}
\end{lemma}

The proof of Lemma \ref{l:lin:con:DS} is provided in the Appendix.

Now we provide necessary and sufficient conditions for the
existence of matrices $\hat E$ and $\hat R_i$, $\forall i\in[1;\widetilde q]$, appearing in
conditions~\eqref{e:lin:con2g} and~\eqref{e:lin:con2h}.

\begin{lemma}\label{l:lin:con:E}
Given $P$ and $E$, there exists a matrix $\hat E$ satisfying~\eqref{e:lin:con2g} if and only if
\begin{IEEEeqnarray}{c}\label{e:lin:con:E}
\im EP\subseteq \im P.
\end{IEEEeqnarray}
\end{lemma}

The proof is recovered from the one of Lemma \ref{l:lin:con:AQ} by substituting $A$, $\hat A$, and $B$ with $E$, $\hat E$, and $0_{n\times m}$, respectively.

\begin{lemma}\label{l:lin:con:R}
Given $P$ and $R_i$, $\forall i\in[1;\widetilde q]$, there exists matrices $\hat R_i$, $\forall i\in[1;\widetilde q]$, satisfying~\eqref{e:lin:con2h} if and only if
\begin{IEEEeqnarray}{c}\label{e:lin:con:R}
\im R_iP\subseteq \im P,
\end{IEEEeqnarray}
for any $i\in[1;\widetilde q]$.
\end{lemma}

The proof is recovered from the one of Lemma \ref{l:lin:con:AQ} by substituting $A$, $\hat A$, and $B$ with $R_i$, $\hat R_i$, $\forall i\in[1;\widetilde q]$, and $0_{n\times m}$, respectively.

Lemmas~\ref{l:lin:con:AQ},~\ref{l:lin:con:DS},~\ref{l:lin:con:E}, and~\ref{l:lin:con:R} provide
necessary and sufficient conditions on $P$ which lead to the
construction of matrices $\hat A$, $\hat D$, $\hat E$, and $\hat R_i$, $\forall i\in[1;\widetilde q]$, together with the matrices
$Q$, $S$ appearing in the definition of the interface function in \eqref{e:lin:int}. The
output matrix $\hat C$ simply
follows by $\hat C=CP$. As we already discussed, the abstract external input matrix can be chosen arbitrarily. For example one can choose $\hat B=I_{\hat n}$ making the abstract system $\hat \Sigma$ fully actuated and, hence, the synthesis problem over it much simpler. One can also choose $\hat B$ as in~\eqref{e:lin:con:B} guaranteeing preservation of all behaviors of $\Sigma$ on $\hat\Sigma$ under extra conditions in \eqref{e:lin:con3}. Lemma 3 in \cite{GP09}, as recalled next, provides necessary and
sufficient conditions on $P$ and $C$ for the existence of $\hat P$,
$G$, and $F$ satisfying \eqref{e:lin:con3a}, \eqref{e:lin:con3b}, and \eqref{e:lin:con3c}.

\begin{lemma}
Consider matrices $C$ and $P$ with $P$ being injective and let $\hat C=CP$. There exists
matrix $\hat P$ satisfying~\eqref{e:lin:con3a}, \eqref{e:lin:con3b}, and \eqref{e:lin:con3c}, for some matrices $G$ and $F$ of appropriate dimensions, if and only
if
\begin{IEEEeqnarray}{c}\label{e:lin:con:BB}
\im P+\ke C=\R^n.
\end{IEEEeqnarray}
\end{lemma}

The conditions~\eqref{e:lin:con:AQ}-\eqref{e:lin:con:R} (resp. \eqref{e:lin:con:AQ}-\eqref{e:lin:con:BB}) complete the characterization of matrix $P$, together with
the system matrices $\{A,B,C,D\}$ leading to the abstract matrices
$\{\hat A,\hat B,\hat C,\hat D\}$, where $\hat B$ can be chosen arbitrarily (resp. $\hat B$ is computed as in \eqref{e:lin:con:B} for the sake of preservation of all behaviors of $\Sigma$ on $\hat\Sigma$ as long as conditions \eqref{e:lin:con3d} and \eqref{e:lin:con3e} are also satisfied).
Note that there always exists an injective matrix $P\in\R^{n\times \hat n}$ that
satisfies the conditions~\eqref{e:lin:con:AQ}-\eqref{e:lin:con:BB}. In the worst-case scenario, we can pick the identity matrix with $\hat n=n$. Of course, we
would like to have the abstraction $\hat\Sigma$ as simple as possible
and, therefore, we should aim at a $P$ with $\hat n$ as small as possible.

We summarize the construction of the abstraction $\hat \Sigma$ in Table~\ref{tb:1}.

\begin{table}[h]
\centering
\begin{tabular}{|p{0.0\columnwidth}l|}
\hline
1.&Compute matrices $M$ and $K$ satisfying~\eqref{e:lin:con1} and~\eqref{e:lin:con11}.\\
2.&Pick an injective $P$ satisfying \eqref{e:lin:con:AQ}-\eqref{e:lin:con:R} (resp. \eqref{e:lin:con:AQ}-\eqref{e:lin:con:BB} only\\& if the computed matrices $\hat P$, $G$, and $F$ satisfy \eqref{e:lin:con3d} and \eqref{e:lin:con3e});\\
3.&Compute $\hat A$ and $Q$ from~\eqref{e:lin:con2a};\\
4.&Compute $\hat D$ and $S$ from~\eqref{e:lin:con2b};\\
5.&Compute $\hat C=CP$;\\
6.&Choose $\hat B$ arbitrarily (resp. $\hat B=[\hat PB\;\; \hat PAG]$);\\
7.&Compute $\widetilde R$, appearing in \eqref{e:lin:int}, from \eqref{e:lin:con:tildeR};\\
8.&Compute $\hat E$ from~\eqref{e:lin:con2g} (resp. $\hat E=\hat PEP$);\\
9.&For any $i\in[1;\widetilde q]$, compute $\hat R_i$ from~\eqref{e:lin:con2h} (resp. $\hat R_i=\hat PR_iP$).\\\hline
\end{tabular}
\caption{Construction of an abstract JLSS $\hat\Sigma$ for a given JLSS $\Sigma$.}\label{tb:1}
\end{table}


\section{An Example}\label{case_study}

Let us demonstrate the effectiveness of the proposed results by
synthesizing a controller for an interconnected system
  consisting of
four JLSS $\Sigma=\mathcal{I}(\Sigma_1,\Sigma_2,\Sigma_3,\Sigma_4)$. The interconnection scheme of $\Sigma$ is illustrated in
Figure~\ref{f:ex1}. 
\begin{figure}[t]
\centering

  \begin{tikzpicture}[auto, node distance=2cm, >=latex]

  \tikzset{%
      block/.style    = {thick,draw, rectangle, minimum height = 3em, minimum width = 3em}
      }

    \node[block] (sys3) at (3,0) {$\Sigma_3$};
    \node[block] (sys4) at (3,-1.5) {$\Sigma_4$};
    \node[block] (sys1) at (0,0) {$\Sigma_1$};
    \node[block] (sys2) at (0,-1.5) {$\Sigma_2$};

    \draw[->] ($(sys3.east)+(0,0.25)$) -- node[near end,below] {$y_{33}$} ++(1.5,0);

    \draw[->] ($(sys4.east)+(0,-.25)$) -- node[near end, above] {$y_{44}$} ++(1.5,0);

    \draw[<-] ($(sys1.west)+(0,-.25)$) -- node[above, near end] {$u_{1}$} ++(-1.5,0);

    \draw[<-] ($(sys2.west)+(0,0.25)$) -- node[below, near end] {$u_{2}$} ++(-1.5,0);

    \draw[->] ($(sys3.east)+(.5,.25)$) |- node[near start, right] {$y_{31}$}  ($(sys1.west)+(-.5,.75)$) |- ($(sys1.west)+(0,.25)$) ;
    \draw[fill] ($(sys3.east)+(.5,.25)$) circle (1pt);

    \draw[->] ($(sys4.east)+(.5,-.25)$) |-node[near start, right] {$y_{42}$} ($(sys2.west)+(-.5,-.75)$) |- ($(sys2.west)+(0,-.25)$);
    \draw[fill] ($(sys4.east)+(.5,-.25)$) circle (1pt);

    \draw[->] ($(sys1.east)$)              -- node[above] {$y_{14}$} 
              ($(sys1.east)+(0.5,0.00)$)  --
              ($(sys4.west)+(-.5,0)$) -- (sys4.west);

    \draw[->] ($(sys2.east)$)              -- node[below] {$y_{23}$} 
              ($(sys2.east)+(0.5,0.00)$)  --
              ($(sys3.west)+(-.5,0)$) -- (sys3.west);

  \end{tikzpicture}
  \caption{The interconnected system
  $\Sigma=\mathcal{I}(\Sigma_1,\Sigma_2,\Sigma_3,\Sigma_4)$.}\label{f:ex1}
\end{figure}
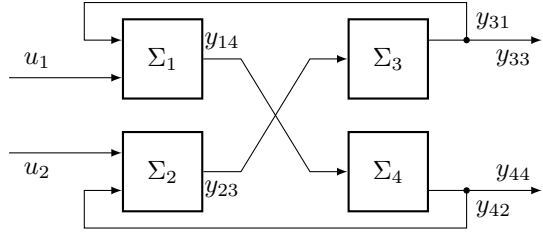%
The system has two outputs and we synthesize a
controller to enforce them to stay approximately (in the $2$nd
moment metric) within the safety constraint
\begin{IEEEeqnarray*}{c}
  S=\intcc{0~5}\times \intcc{0~5}.
\end{IEEEeqnarray*}
We refer the interested readers to the explanation provided before \cite[Remark 5.5]{majid8} or to \cite[Subsection 5.1]{majid10} concerning the interpretation of the satisfaction of a safety constraint in the moment over the concrete stochastic systems.

In designing a controller for $\Sigma$ we proceed as follows. In the first step, we compute abstractions
$\hat \Sigma_i$ of the individual subsystems to obtain an abstraction
$\hat \Sigma=\mathcal{I}(\hat \Sigma_1,\hat \Sigma_2,\hat
\Sigma_3,\hat \Sigma_4)$ of the interconnected system $\Sigma$. The
interconnection scheme changes for $\hat \Sigma$ (see
Remark~\ref{r:cutConnection}) and the abstract system is given by two 
identical independent interconnected systems $\hat\Sigma_{14}=\mathcal{I}(\hat \Sigma_1,\hat
\Sigma_4)$ and $\hat \Sigma_{23}=\mathcal{I}(\hat
\Sigma_2,\hat \Sigma_3)$.
The abstract
system $\hat \Sigma$ is illustrated in Figure~\ref{f:ex2}. 
\begin{figure}[t]
\centering

  \begin{tikzpicture}[auto, node distance=2cm, >=latex]

  \tikzset{%
      block/.style    = {thick,draw, rectangle, minimum height = 3em, minimum width = 3em}
      }

    \node[block] (sys3) at (3,0) {$\hat\Sigma_3$};
    \node[block] (sys4) at (3,-1.5) {$\hat\Sigma_4$};
    \node[block] (sys1) at (0,0) {$\hat\Sigma_1$};
    \node[block] (sys2) at (0,-1.5) {$\hat\Sigma_2$};

    \draw[->] ($(sys3.east)+(0,0)$) -- node[near end,below] {$\hat y_{33}$} ++(1.5,0);

    \draw[->] ($(sys4.east)+(0,0)$) -- node[near end, above] {$\hat y_{44}$} ++(1.5,0);

    \draw[<-] ($(sys1.west)+(0,0)$) -- node[above, near end] {$\hat u_{1}$} ++(-1.5,0);

    \draw[<-] ($(sys2.west)+(0,0)$) -- node[below, near end] {$\hat u_{2}$} ++(-1.5,0);

    \draw[->] ($(sys1.east)$)              -- node[above] {$\hat y_{14}$} 
              ($(sys1.east)+(0.5,0.00)$)  --
              ($(sys4.west)+(-.5,0)$) -- (sys4.west);

    \draw[->] ($(sys2.east)$)              -- node[below] {$\hat y_{23}$} 
              ($(sys2.east)+(0.5,0.00)$)  --
              ($(sys3.west)+(-.5,0)$) -- (sys3.west);

  \end{tikzpicture}
  \caption{The abstract interconnected system
  $\hat\Sigma=\mathcal{I}(\hat\Sigma_1,\hat\Sigma_2,\hat\Sigma_3,\hat\Sigma_4)$.}\label{f:ex2}
\end{figure}
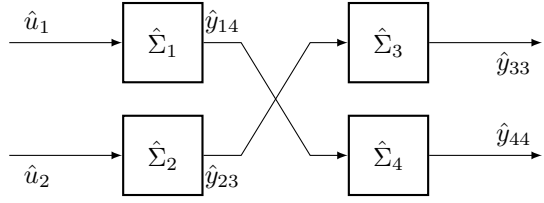
In the
second step, we determinize the stochastic systems
$\hat\Sigma_{14}$ and $\hat
\Sigma_{23}$ by neglecting the diffusion and reset terms. We obtain
two identical deterministic control systems $\tilde\Sigma_{14}$ and $\tilde
\Sigma_{23}$. We show that $\tilde
\Sigma_{i}$ is an abstraction of $\hat \Sigma_i$, $i\in\{14,23\}$ by computing an SSF-M$_2$
function from $\tilde \Sigma_{i}$ to $\hat \Sigma_i$.
In the third step, we fix a sampling time \mbox{$\tau>0$}
and use the MATLAB Toolbox MPT~\cite{MPT3} to
synthesize a safety controller that enforces the safety constraints on
$\tilde \Sigma=\mathcal{I}(\tilde \Sigma_{14},\tilde\Sigma_{23})$ at all sampling times $k\tau$, $k\in\N$.
In the final step, we refine the computed controller for $\tilde \Sigma$ to a
controller for $\Sigma$. We use Theorem~\ref{theorem1} to establish a bound on the distance in the $2$nd moment metric between
the output trajectories of $\Sigma$ and~the safe set $S$.

\subsection{The interconnected system}
Let us consider the system
illustrated in
Figure~\ref{f:ex1}. The subsystems $\Sigma_1$ and $\Sigma_2$ are
double integrators and $\Sigma_3$ and $\Sigma_4$ are autonomous triple
integrators. All systems are affected by a scalar Brownian motion and a Poisson process. For $j\in\{1,2\}$ the system matrices are given by
\begin{IEEEeqnarray*}{l}
A_j
=
\begin{bmatrix}
 0 &  1\\
 2 &  0 
\end{bmatrix},~
B_j=
\begin{bmatrix}
 0\\
 1
\end{bmatrix},~
C^T_{j}=
\begin{bmatrix}
 1\\
 0
\end{bmatrix},~
E_j=0.4I_2,~R_j=0.1I_2,
\end{IEEEeqnarray*}
and for $i\in\{3,4\}$ by 
\begin{IEEEeqnarray*}{l}
A_i
=
\begin{bmatrix}
 \phantom{-}0 & \phantom{-}  1 &\phantom{-} 0\\
 \phantom{-}0 &\phantom{-}  0 &\phantom{-} 1\\
 -24 & -26 & -9
\end{bmatrix},~
B_i=0,~
C^T_{i}=
\begin{bmatrix}
 1\\
 0\\
 0
\end{bmatrix},~
E_i=0.4I_3,~R_i=0.1I_3.
\end{IEEEeqnarray*}%
The rate of the Poisson
process $P_t$ is $\lambda=4.2$. The output of $\Sigma_1$ (resp. $\Sigma_2$) is connected to the internal input of
$\Sigma_{4}$ (resp. $\Sigma_3$) and the output of $\Sigma_{3}$ (resp. $\Sigma_4$) connects to the internal input of
$\Sigma_{1}$ (resp. $\Sigma_2$). The output functions
$h_{ij}(x_i)=C_{ij}x_i$ are determined by
\begin{IEEEeqnarraybox}{l}
C_{ii}=C_{i(i-2)}=
\begin{bmatrix} 1& 0& 0 \end{bmatrix}
\end{IEEEeqnarraybox} for $i\in\{3,4\}$,  
$C_{23}=C_{14}=
\begin{bmatrix} 1& 0 \end{bmatrix}$
and  $h_{ij}\equiv0$ for the remaining $i,j\in\intcc{1;4}$. Correspondingly, the internal input matrices
are given by 
\begin{IEEEeqnarray*}{c}
D_{41}= D_{32}=
\begin{bmatrix}
 0\\
 -d\\
 5d
\end{bmatrix}
,\;
D_{j(j+2)}=
\begin{bmatrix}
 0\\
 d
\end{bmatrix},\;d\neq 0,\; j\in\{1,2\}.
\end{IEEEeqnarray*}
Subsequently, we use $C_1=C_{14}$, $C_2=C_{23}$, \mbox{$C_i=C_{ii}$}, \mbox{$i\in\{3,4\}$}, 
$D_1=D_{13}$, $D_2=D_{24}$, $D_3=D_{32}$, $D_4=D_{41}$, and denote the JLSS by \mbox{$\Sigma_i=(A_i,B_i,C_i,D_i,E_i,R_i)$}.

\subsection{The abstract subsystems}
In order to construct an abstraction for
$\mathcal{I}(\Sigma_1,\Sigma_2,\Sigma_3,\Sigma_4)$ we construct an
 abstraction $\hat \Sigma_i$ of each individual
subsystem $\Sigma_i$, $i\in\{1,2,3,4\}$. We begin with  $i\in\{1,2\}$ and follow the steps outlined
in Table~\ref{tb:1}. First, we fix $\widehat\kappa=3$ and solve an appropriate
LMI (see Lemma~\ref{l:lmi}) to determine the matrices $M_i$ and $K_i$ so
that~\eqref{e:lin:con1} and ~\eqref{e:lin:con11} hold. We obtain
\begin{IEEEeqnarray*}{c}
M_{i}
=
\begin{bmatrix}
1.68&    0.4\\
0.4&    0.23
\end{bmatrix}
,\;
K_i^T=
\begin{bmatrix}
 -9\\
 -4
\end{bmatrix}.
\end{IEEEeqnarray*}
We continue with step 2. and determine 
\begin{IEEEeqnarray*}{c}
P_i^T=\begin{bmatrix} 1& -2\end{bmatrix},
\end{IEEEeqnarray*}
so that \eqref{e:lin:con:AQ}-\eqref{e:lin:con:BB} hold. The matrices $\hat P_i$,
$F_i$, and $G_i$ that~\eqref{e:lin:con3b}-\eqref{e:lin:con3e} hold, follow by
$\hat P_i=\begin{bmatrix} 1& 0\end{bmatrix}$,
$G_i^T=\begin{bmatrix}0 & 2\end{bmatrix}$, and $F_i=\begin{bmatrix}1 &
0\end{bmatrix}$.
We continue with steps 3.-8.~and get the scalar abstract JLSS subsystems 
$\hat \Sigma_i$, $i\in\{1,2\}$ with
\begin{IEEEeqnarray*}{c}
\hat A_i=-2,\; \hat B_i=1,\; \hat D_i=0,\;  \hat C_i=1,\;  \hat
E_i=0.4,\;  \hat R_i=0.1.
\end{IEEEeqnarray*}
Simultaneously, we compute $Q_i=2$
and $S_i=-d$. As already discussed in Remark~\ref{r:cutConnection},
$D_i\in \im B_i$ and we can choose $\hat D_i=0$. It follows that the subsystems $\hat \Sigma_i$,
$i\in\{1,2\}$, are not affected by internal inputs, which implies that
the interconnection between $\Sigma_3$ (resp. $\Sigma_4$) and $\Sigma_1$ (resp. $\Sigma_2$) is
absent on the abstract interconnected system $\hat \Sigma$; compare also
Figure~\ref{f:ex1} and Figure~\ref{f:ex2}.

We continue with the construction of $\hat \Sigma_i$ for $i\in\{3,4\}$. 
We repeat the procedure and obtain
\begin{IEEEeqnarray*}{c'c}
M_{i}
=
\begin{bmatrix}
  6.924&    3.871&    0.468\\
  3.871&    2.534&    0.315\\
  0.468&    0.315&    0.054
\end{bmatrix},&K_i=0.
\end{IEEEeqnarray*}
In step 2., we compute 
\begin{IEEEeqnarray*}{c't'c}
P_i^T=\begin{bmatrix} 1& -2 & 4 \\ 1 & -3 & 9\end{bmatrix},
\end{IEEEeqnarray*}
so that \eqref{e:lin:con:AQ}-\eqref{e:lin:con:BB} hold. The equations~\eqref{e:lin:con3b}-\eqref{e:lin:con3e} are satisfied by
\begin{IEEEeqnarray*}{c't'c}
\hat P_i=\tfrac{1}{6}\begin{bmatrix} 0& -9 & -3 \\ 0 & \phantom{-}4 & \phantom{-}2\end{bmatrix},
\end{IEEEeqnarray*}
$G_i^T=\begin{bmatrix}1 & 0&0\end{bmatrix}$, and
$F_i=\tfrac{1}{6}\begin{bmatrix}6 & 5&1\end{bmatrix}$.
We follow steps 3.-8.~and get the 2D abstract JLSS subsystems
$\hat \Sigma_i$, $i\in\{3,4\}$, where
\begin{IEEEeqnarray*}{c}
\hat A_i=\begin{bmatrix}-2& 0\\ 0 & -3\end{bmatrix},
\hat B_i=\begin{bmatrix}12\\ -8\end{bmatrix}, 
\hat D_i=d\begin{bmatrix}-1 \\ \phantom{-}1 \end{bmatrix},
\hat C_i=\begin{bmatrix}1 & 1 \end{bmatrix},
\end{IEEEeqnarray*}
with the diffusion and reset terms again given by $\hat E_i=0.4I_2$
and $\hat R_i=0.1I_2$.
Moreover, $Q_i=0$ and $S_i=0$. 

For all $i\in \{1,2,3,4\}$, equations~\eqref{e:lin:con1},~\eqref{e:lin:con11}, and \eqref{e:lin:con2} hold. Hence, Theorem~\ref{t:lin:suf} applies and we see that~$V_i(x_i,\hat x_i)=(x_i-P_i\hat x_i)^T M_i(x_i-P_i\hat x_i)$ is
an SSF-M$_2$ function from~$\hat\Sigma_i$ to~$\Sigma_i$ for all
$i\in\intcc{1;4}$. Moreover, \eqref{e:lin:con3} holds and Theorem~\ref{t:B} implies that all the behaviors of $\Sigma_i$ are preserved on $\hat \Sigma_i$. Following the proof of Theorem~\ref{t:lin:suf}, we see that
the interface function for $i\in\{1,2\}$ follows by~\eqref{e:lin:int} as
\begin{align}\label{e:lin:ex:interf}
\small
\nu_{i\hat\nu_i}(x_i,\hat x_i,\hat u_i,\hat w_i)
=
K_i(x_i-P_i\hat x_i)
-2\hat x_i
-2.5\hat u_i
-d\hat w_i,
\end{align}
and $\nu_{i\hat \nu_i}\equiv 0$ for $i\in\{3,4\}$. Here we
used~\eqref{e:lin:con:tildeR} to compute $\widetilde R_i=-2.5$ for $i\in\{1,2\}$.
Although the internal input matrices for $\Sigma_1$ and $\Sigma_2$ are zero, the
internal inputs $\hat w_1= \hat y_3$ and $\hat w_2=\hat y_4$ still appear in the interface function.
As provided in the proof of Theorem \ref{t:lin:suf} and by fixing $\pi=1$, the $\mathcal{K}_\infty$ functions for
$i\in\{1,2\}$ and $j\in\{3,4\}$ are given by
\begin{IEEEeqnarray*}{l;l;l;l}
\alpha_i(s)=s,&\eta_i(s)=2s,& \rho_{i\mathrm{ext}}(s)=0.16s, &
\rho_{i\mathrm{int}}(s)=1.3d^2s, \\
\alpha_j(s)=s,&\eta_j(s)=2s,& \rho_{j\mathrm{ext}}(s)=150s, &
\rho_{j\mathrm{int}}(s)=7.9d^2s,
\end{IEEEeqnarray*}
for any $s\in\R_{\ge0}$.

\subsection{The interconnected abstraction}

We now proceed with Theorem~\ref{t:ic} to construct a stochastic simulation function form
$\hat \Sigma$ to $\Sigma$. We start by 
checking the Assumption~\ref{assumption1}. Note that $\rho_{i\mathrm{int}}$
satisfies the triangle inequality and we use Remark~\ref{r:triangle} to see that Assumption~\ref{assumption1}
holds for $\gamma_i(s)=s$, $\widetilde\lambda_i=2$, and $\delta_{ij}$ are given by
\begin{IEEEeqnarray*}{c'c}
\Delta=d^2 \begin{bmatrix}
0 & 0 & 1.3 & 0\\
0 & 0 & 0 & 1.3\\
0 & 7.9 & 0 & 0\\
7.9 & 0 & 0 & 0
\end{bmatrix}.
\end{IEEEeqnarray*}
Additionally, we require the existence of a vector $\mu\in\R_{>0}^4$ satisfying~\eqref{e:SGcondition}, which is the case if and only if the
spectral radius of $\Delta$ is strictly less than one, i.e.,
$1/2\sqrt{1.3\times 7.9}d^2<1$, which holds for example for $d=1/2$. One can choose the vector $\mu$
as $\mu=[2\;2\;1\;1]$ and, hence, it follows that 
\begin{IEEEeqnarray*}{c}
V(x,\hat x)=\sum_{i=1}^2 2V_i(x_i,\hat x_i) +\sum_{i=3}^4 V_i(x_i,\hat x_i),
\end{IEEEeqnarray*}
is an SSF-M$_2$ from $\mathcal{I}(\hat \Sigma_1,
\hat \Sigma_2,\hat \Sigma_3,\hat \Sigma_4)$ to
$\mathcal{I}(\Sigma_1,\Sigma_2,\Sigma_3,\Sigma_4)$ where the interface function
follows from~\eqref{e:lin:ex:interf}. Following the proof of
Theorem \ref{t:ic}, we see that $V$ satisfies~\eqref{e:sf:1}
with $\alpha(s)=s$ and~\eqref{inequality1} with $\eta(s)=1.35s$,
$\rho_{\mathrm{ext}}(s)=150s$, and $\rho_{\mathrm{int}}\equiv0$. Here, we
computed $\eta$ and $\rho_{\mathrm{ext}}$ according to~\eqref{lambda}
and~\eqref{rho}. 
Subsequently, we design a controller for $\Sigma$ via the abstraction $\hat
\Sigma$. We restrict external inputs for $\hat \Sigma_3$ and $\hat \Sigma_4$ to
zero, so that we can set $\rho_{j\mathrm{ext}}\equiv 0$, $j\in\{3,4\}$. As a
result $\rho_{\mathrm{ext}}$ reduces to $\rho_{\mathrm{ext}}(s)=0.16s$, $\forall s\in\R_{\geq0}$, and we
use Theorem~\ref{theorem1} in combination with Remark~\ref{bound} to derive the inequality 
\begin{align}\label{e:ex:bound}
\EE[\Vert\zeta_{a\nu}(t)-\hat \zeta_{\hat a\hat \nu}(t)\Vert^2]
\le
\mathsf{e}^{-1.35t}\EE[V(a,\hat a)]
+0.12\EE[\Vert\hat \nu\Vert_\infty^2].
\end{align}

\subsection{The deterministic system and the controller}
The synthesis of the safety controller is based on a deterministic system $\tilde\Sigma$
which results from $\hat\Sigma$ by omitting the diffusion and reset terms. In
particular, we determinize the identical systems $\hat
\Sigma_{14}=\mathcal{I}(\hat\Sigma_1,\hat\Sigma_4)$ and  $\hat
\Sigma_{23}=\mathcal{I}(\hat\Sigma_2,\hat\Sigma_3)$  and obtain for $i\in\{14,23\}$ the
systems
\begin{IEEEeqnarray*}{c}
\tilde\Sigma_{i}: \left\{
  \begin{IEEEeqnarraybox}[][c]{rCl}
  \dot{\tilde \xi}_i(t)&=&
    \begin{bmatrix}
    -2 & \phantom{-}0 & -d \\
     0 & -3 & \phantom{-}d \\
     0 & \phantom{-}0 & -2 
    \end{bmatrix}
  \tilde \xi_i(t)+
    \begin{bmatrix}
     0 \\
     0 \\
     1 
    \end{bmatrix}
  \tilde \nu_i(t),\\
    \tilde\zeta_i(t)&=&\begin{bmatrix} 1& 1& 0\end{bmatrix}\tilde\xi_i(t).
  \end{IEEEeqnarraybox}
  \right.
\end{IEEEeqnarray*}
We compute an SSF-M$_2$ function $\hat V(\hat x,\tilde x)=[\hat x; \tilde x]^T \hat M[\hat x; \tilde x]$ from $\tilde
\Sigma=\mathcal{I}(\tilde\Sigma_{14},\tilde \Sigma_{23})$ to $\hat \Sigma$,
by solving an appropriate LMI. The matrix $\hat M$ results in
\begin{IEEEeqnarray*}{c}
\hat M
=
\left[
\begin{IEEEeqnarraybox}[][c]{r'r'r'r}
m_1& 0 & -m_2& 0\\
0 & m_1 & 0 &-m_2\\
-m_2^T & 0 & m_3& 0\\
0 &-m_2^T & 0& m_3
\end{IEEEeqnarraybox}
\right]
\end{IEEEeqnarray*}
with 
\begin{IEEEeqnarray*}{l}
m_1=
\begin{bmatrix}
  1.1400 &   1.3072 &   0.0052\\
  1.3072 &   1.6968 &   0.0228\\
  0.0052 &   0.0228 &   0.0104
\end{bmatrix},~
m_2=
\begin{bmatrix}
 1.1437&   1.3112&   0.0060\\
 1.3365&   1.7181&   0.0218\\
 0.0089&   0.0230&   0.0085
\end{bmatrix},~
m_3=
\begin{bmatrix}
   1.1793 &   1.3649 &   0.0081\\
   1.3649 &   1.7631 &   0.0224\\
   0.0081 &   0.0224 &   0.0079
\end{bmatrix}.
\end{IEEEeqnarray*}
The associated $\mathcal{K}_\infty$ functions for $\hat V$ are
given by \mbox{$\alpha(r)=r$}, \mbox{$\eta(r)=0.82r$}, $\rho_{\mathrm{ext}}(r)=0.32r$, and
$\rho_{\mathrm{int}}\equiv 0$. Again we use Theorem~\ref{theorem1} and Remark~\ref{bound} to establish
\begin{IEEEeqnarray}{c}\label{e:ex:bound1}
\EE[\Vert\hat\zeta_{\hat a\tilde \nu}(t)-\tilde \zeta_{\tilde a\tilde \nu}(t)\Vert^2]
\le
\mathsf{e}^{-0.82t}\EE[\hat V(\hat a,\tilde a)]
+0.4\Vert\tilde \nu\Vert_\infty^2.\IEEEeqnarraynumspace
\end{IEEEeqnarray}

Next we design a safety controller to restrict the
output $\tilde y\in\R$ of $\tilde \Sigma_i$, $i\in\{14,23\}$ to $\intcc{0~5}$.
Additionally, to control the mismatch between the trajectories of $\Sigma$ and
$\tilde \Sigma$, we limit the inputs to $\tilde u\in\intcc{-1~1}^2$.
We fix the
sampling time to $\tau=0.1\;$ secs and use the
MATLAB Toolbox MPT~\cite{MPT3} to compute a safety controller
$K:\R^6\to 2^{\intcc{-1~1}^2}$, which when applied in a sample-and-hold
manner to $\tilde\Sigma$ enforces the constraints at the sampling instances
$t=k\tau$, $k\in\N$. A part of the domain of the
controller, which restricts the initial states of $\tilde \Sigma$ is
illustrated in Figure~\ref{f:cdom}. Note that $K$ is a set-valued map that
provides, for each state $\tilde x$ in the domain of $K$, possibly a set of admissible inputs
$K(x)\subseteq \intcc{-1~1}^2$.
\begin{figure}
\centering
\includegraphics[width=9cm]{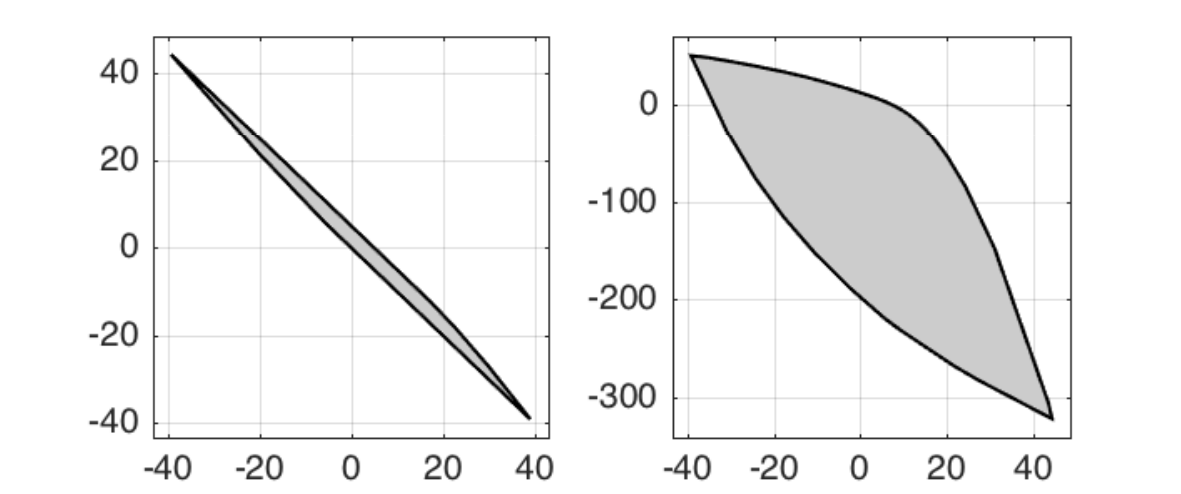}
\caption{Part of the domain of the safety controller. The left figure shows the projection
on $\tilde x_1$ and $\tilde x_2$. The right figure shows the projection on
$\tilde x_2$ and~$\tilde x_3$.}\label{f:cdom}
\end{figure}
%

\subsection{Input trajectory generation and performance guarantees}

We use the closed-loop system consisting of $\tilde \Sigma$ and $K$ to generate
input trajectories for $\Sigma$. Let $(\tilde\xi,\tilde\zeta,\tilde \nu)$ be a
trajectory of $\tilde \Sigma$ that satisfies $K$, i.e.,
$\tilde\nu$ is constant on the intervals $\tau\intco{k,(k+1)}$, $k\in\N$, and satisfies
$\tilde\nu(k\tau)\in K(\tilde\xi(k\tau))$ for all $k\in\N$. We use the interface~\eqref{e:lin:ex:interf} to compute the input
trajectory $\nu$ for $\Sigma$. Using the bounds in \eqref{e:ex:bound} and \eqref{e:ex:bound1}, the
overall estimate between output trajectories of $\tilde \Sigma$ and $\Sigma$ follows to
\begin{align}\notag
\left(\EE[\Vert\zeta_{a\nu}(t)-\tilde \zeta_{\tilde a\tilde \nu}(t)\Vert^2]\right)^{\frac{1}{2}}
&\le
\left(\EE[\Vert\zeta_{a\nu}(t)-\hat \zeta_{\hat a\tilde \nu}(t)\Vert^2]\right)^{\frac{1}{2}}+\left(\EE[\Vert\hat\zeta_{\hat a\tilde\nu}(t)-\tilde \zeta_{\tilde a\tilde \nu}(t)\Vert^2]\right)^{\frac{1}{2}}\\\label{final_bound}&\le
\mathsf{e}^{-0.67t}\EE[V(a,\hat a)]^{\frac{1}{2}}+\mathsf{e}^{-0.41t}\EE[\hat V(\hat a,\tilde a)]^{\frac{1}{2}}+\Vert\tilde \nu\Vert_\infty.
\end{align}

We show some
simulation results of the controlled system in
Figure~\ref{f:simres}. The initial state of $\Sigma$ is fixed as $a=\intcc{1;-1;-5;1;-1;-5;1;-2;1;-2}$. We determine
the initial state for $\hat \Sigma$ as well as $\tilde
\Sigma$ as the vector $\tilde a\in \R^6$ lying in the domain of the
controller and minimizing $V(a,\tilde a)$ which is $\tilde a=\intcc{1.44;-0.69;1.44;-0.69;1;1}$.
We randomly pick the input $\tilde
\nu(k\tau)$ in $K(\tilde \xi(k\tau))$. In the top two plots of the figure, we see a realization of the observed process $\zeta_1$ (resp. $\zeta_2$) and $\hat\zeta_1$ (resp. $\hat\zeta_2$) of $\Sigma$ and $\hat\Sigma$, respectively. On the middle plot, we show the corresponding evolutions of the refined input signals $\nu_1$ and $\nu_2$ for $\Sigma$. On the 2nd plot from bottom, we show the square root of the average value (over 1000 experiments) of the squared distance in time of the output trajectory of $\Sigma$ to the one of $\hat\Sigma$, namely, $\Vert\zeta_{a\nu}-\hat\zeta_{\tilde a\tilde\nu}\Vert^2$. The solid black curve denotes
the error bound given by the right-hand-side of \eqref{e:ex:bound}. On the bottom part, we show the square root of the average value (over 1000 experiments) of the squared distance in time of the output trajectory of $\Sigma$ to the set $S$, namely, $\Vert \zeta_{a\nu}(t)\Vert_S$.
Notice that the square root of this empirical
(averaged) squared distances is significantly lower than the computed bound given by the right-hand-side of~\eqref{final_bound}, as expected since the stochastic simulation functions can lead to conservative bounds. (One can improve the bounds by seeking
optimized stochastic simulation functions.)

\begin{figure}
\begin{center}
\begin{tikzpicture}
\node[inner sep=0] at (0,0) {\includegraphics[width=8cm]{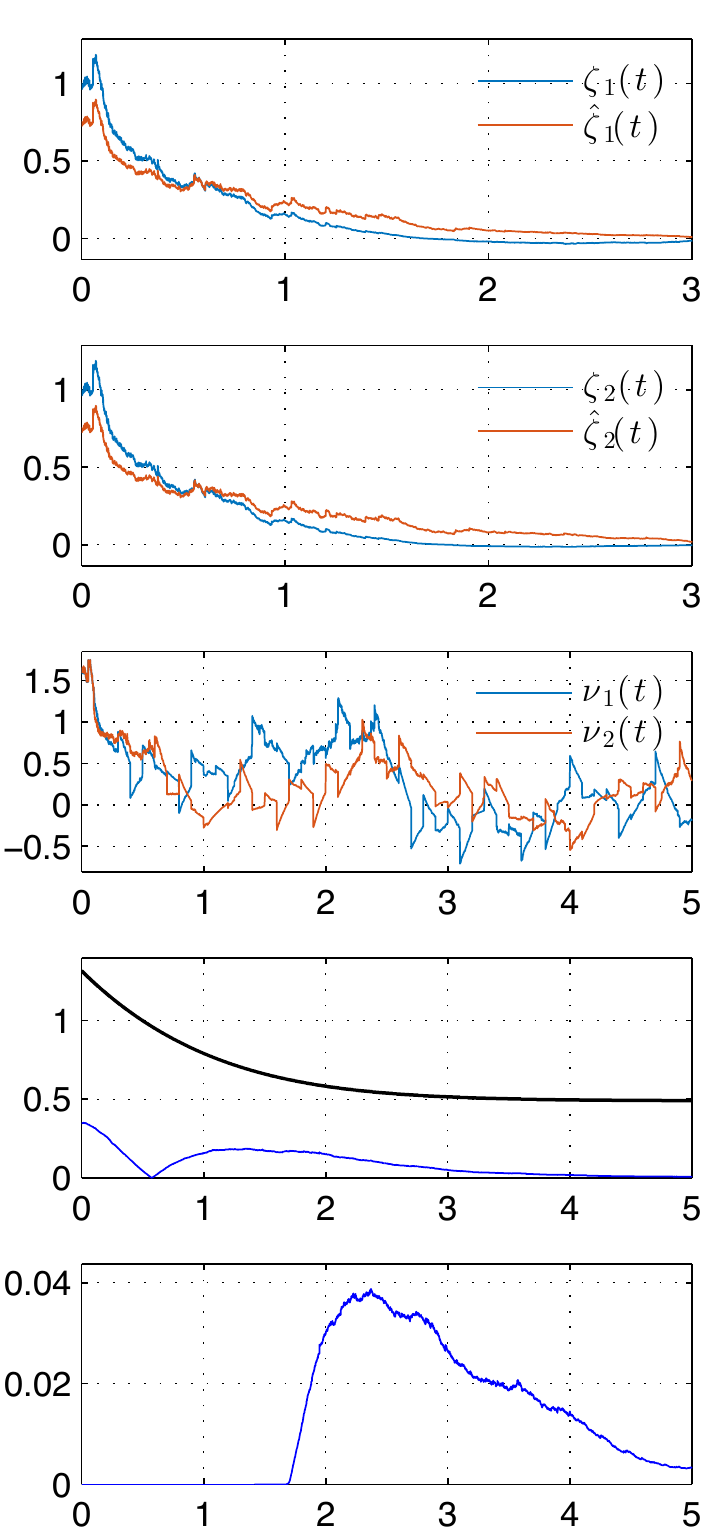}};
\node[rotate=90] at (-4.2,-3.5) {\small $\sqrt{\EE[\Vert\zeta(t)-\hat\zeta(t)\Vert^2]}$};
\node[rotate=90] at (-4.2,-7) {\small $\sqrt{\EE[\Vert\zeta(t)\Vert_S^2]}$};
\node[inner sep=0] at (.2,-8.9) {\small $t\;[\mathrm{sec}]$};
\end{tikzpicture}
\caption{Top two plots: One realization of $\zeta_1$ (resp. $\zeta_2$)
(\protect\tikz[baseline=-.5ex]{\protect\draw[very thick,MATLABblue] (0,0) -- (0.2,0);})
and $\hat \zeta_1$ (resp. $\hat \zeta_2$) (\protect\tikz[baseline=-.5ex]{\protect\draw[very thick,MATLABred] (0,0) -- (0.2,0);}). The middle plot: the corresponding realization of external inputs $\nu_1$
(\protect\tikz[baseline=-.5ex]{\protect\draw[very thick,MATLABblue] (0,0) -- (0.2,0);})
and $\nu_2$ (\protect\tikz[baseline=-.5ex]{\protect\draw[very thick,MATLABred] (0,0) -- (0.2,0);}) of $\Sigma$. The 2nd plot from bottom: Square root of the average values (over 1000 experiments)
of the squared distance of the output trajectory of $\Sigma$ to the one of $\hat\Sigma$. The solid black line indicates the error bound given by the
right-hand-side of~\eqref{e:ex:bound}. Bottom plot: Square root of the average values (over 1000 experiments)
of the squared distance of the output trajectory of $\Sigma$ to the safe set $S$.}  \label{f:simres}\end{center}
\end{figure}

\section{Summary}
In this paper we proposed a compositional framework for the construction of infinite approximations of interconnected stochastic hybrid systems by leveraging some small-gain type conditions. We introduced a new notion of stochastic simulation functions to quantify the error between the stochastic hybrid systems and their approximations. In comparison with the similar notion in \cite{julius1}, our proposed notion of stochastic simulation functions is computationally more tractable for stochastic hybrid systems with inputs. Moreover, we provided a constructive approach on the construction of those infinite approximations for a class of stochastic hybrid systems, namely, jump linear stochastic systems. Finally, we illustrated the effectiveness of the results by constructing an infinite approximation of an interconnection of four jump linear stochastic systems in a compositional manner. We employed the constructed
approximation as a substitute in the controller synthesis scheme to enforce a safety constraint on the concrete interconnected system, would not have been possible to enforce without the use of the approximation.

\section*{Appendix}
\begin{proof}[Proof of Lemma \ref{lem:KL}]
Lemma \ref{lem:KL} is an extension of Lemma 4.4 in \cite{LSW96} and the proof follows similar ideas. The proof includes two steps. We first show that the set $[0,s_0]$, $s_0 \Let \eta^{-1}(2g)$, is forward invariant, i.e., if $y(t_0) \in [0,s_0]$, then $y(t) \in [0,s_0]$ for all $t \ge t_0$. For the sake of contradiction, suppose the trajectory $y$ visits $[0,s_0]$ and then later leaves it. Due to the continuity of $y$, this implies that there exist a time instance $t> t_0$ and positive value $\eps >0$ such that $y(t_0) = s_0$ and $y(t) = s_0 + \eps$, and $y(\tau) \geq s_0$ for all $\tau \in [t_0,t]$. In view of the lemma hypothesis, we then have 
\begin{align*}
	0 < \eps = y(t) - y(t_0) \le \int_{t_0}^t [- \eta(y(\tau)) + g] \le 0, 
\end{align*}
which concludes the first step. In the second step, we assume that $y(0) > s_0$. Consider the function $\kappa: \R_{>0} \ra \R$ defined as
\begin{align*}
	\kappa(s) \Let\int^{s}_{1} {-\diff r \over \min\{\eta(r),r\}}. 
\end{align*}
Let $t_s$ be the first time that the process $y$ reaches $s_0$, i.e., $t_s \Let \inf\{t \ge 0 : y(t) \le s_0\}$.\footnote{By convention, $\inf \emptyset = \infty$.} In the following we show that the function 
\begin{align}\label{KLfunct}
	\vartheta(r,t) \Let \left\{\begin{array}{cc} 
	\kappa^{-1}\big( \kappa(r) + t/2\big), & r > 0 \\ 0 & r = 0,
	\end{array}\right. 
\end{align}
is indeed the desired $\KL$ function for the lemma assertion. Note that for all $t \in [0, t_s]$, we have $\eta\big(y(t)\big) \ge 2g$, and that we have
\begin{align*}
	\kappa\big(y(t)\big) -\kappa\big(y(0)\big) & = \int\limits_{y(0)}^{y(t)} {-\diff(y(\tau)) \over \min\{\eta\big(y(\tau)\big),y(\tau)\}} 
	 \ge \int\limits_{0}^{t} { \eta\big(y(\tau)\big) - g \over \min\{\eta\big(y(\tau)\big),y(\tau)\}}\diff \tau  
	 \ge \int\limits_{0}^{t} {{1\over 2} \eta\big(y(\tau)\big) \over\eta\big(y(\tau)\big)}\diff\tau \ge {t\over 2}.
\end{align*}
The above observation together with the fact that the function $\kappa$ is strictly decreasing on $(0, \infty)$ imply that 
\begin{align*}
	y(t) \le \kappa^{-1}\Big(\kappa\big(y(0)\big) + t/2\Big), \qquad \forall t \in [0,t_s].
\end{align*}
Note that $\lim_{s \da 0}\kappa(s) = \infty$, and since $\kappa$ is strictly decreasing on $(0,\infty)$, the function $\vartheta(r,t)$ defined in \eqref{KLfunct} is a $\Kinf$ function in the first argument for each $t$, and decreasing with respect to the second argument for each nonzero $r$. As such, the function $\vartheta(r,t)$ is a $\KL$ function. Combining the results of the two steps concerning the intervals $[0,t_s]$ and $(t_s,\infty)$ concludes the desired assertion. 
\end{proof}

\begin{proof}[Proof of Theorem \ref{theorem1}]
For any time instances $t \ge t_0 \ge 0$, any $\hat\nu(t)\in\R^{\hat m}$, any $\hat\omega(t)\in\R^p$, and any random variable $a$ and $\hat{a}$ that are $\sigalg_0$-measurable, there exists $\nu(t)\in\R^m$ such that for all $\omega(t)\in\R^p$, one obtains
	\begin{align}
		\notag \EE &\left[ V(\xi_{a\nu\omega}(t),\hat\xi_{\hat a\hat\nu\hat\omega}(t)) \right]
		 = \EE \left[V\big(\xi_{a\nu\omega}(t_0),\hat\xi_{\hat a\hat\nu\hat\omega}(t_0)\big) + \int_{t_0}^{t} \mathcal{L} V(\xi_{a\nu\omega}(s),\hat\xi_{\hat a\hat\nu\hat\omega}(s))\diff s\right]
		\\ \notag & \le \EE \left[ V\big(\xi_{a\nu\omega}(t_0),\hat\xi_{\hat a\hat\nu\hat\omega}(t_0)\big)\right] 
		 +\EE\bigg[\int_{t_0}^{t} -\eta\Big( V\big(\xi_{a\nu\omega}(s),\hat\xi_{\hat a\hat\nu\hat\omega}(s)\big)\Big)
		+ \rho_{\mathrm{ext}}(\Vert \hat\nu(s)\Vert^k)+\rho_{\mathrm{int}}(\Vert \omega(s)-\hat \omega(s)\Vert^k) \diff s \bigg]
		\\ \notag & \le \EE \left[ V\big(\xi_{a\nu\omega}(t_0),\hat\xi_{\hat a\hat\nu\hat\omega}(t_0)\big)\right] 
		+\int_{t_0}^{t} -\eta\Big(\EE\left[ V\big(\xi_{a\nu\omega}(s),\hat\xi_{\hat a\hat\nu\hat\omega}(s)\big)\right] \Big)
		+\EE\Big[\rho_{\mathrm{ext}}(\Vert \hat\nu\Vert^k_\infty)+\rho_{\mathrm{int}}(\Vert \omega-\hat \omega\Vert^k_\infty)\Big]\diff s,
	\end{align}
	where the first equality is an application of the It\^{o}'s formula for jump diffusions thanks to the polynomial rate of the function $V$ \cite[Theorem 1.24]{ref:Oks-jump}, and the last inequality follows from Jensen's inequality due to the convexity assumption on the function $\eta$ \cite[p. 310]{oksendal}. Let us define the process $y(t) \Let  \EE\big[V(\xi_{a\nu\omega}(t),\hat\xi_{\hat a\hat\nu\hat\omega}(t))\big]$. Note that in view of the It\^{o}'s formula, the process $y(\cdot)$ is continuous provided that the solution processes $\xi_{a\nu\omega}$ and $\hat\xi_{\hat a\hat\nu\hat\omega}$ have finite moments. This is indeed the case under our model setting in Definition \ref{Def_control_sys}, in particular due to the Lipschitz continuity of functions $f, \sigma, r,\hat f, \hat\sigma,\hat r$ \cite[1.19]{ref:Oks-jump}.
Therefore, the process $y(t)$ meets all the required assumptions of Lemma~\ref{lem:KL}, implying that there exists a $\KL$ function $\vartheta$ such that 
\begin{align}
	\label{ineq2} \EE[ V(\xi_{a\nu\omega}(t),&\hat\xi_{\hat a\hat\nu\hat\omega}(t)) ]\le \vartheta\big(\EE[V(a,\hat a)], t\big) 
	+ \eta^{-1}\Big(2\EE\big[\rho_{\mathrm{ext}}(\Vert \hat\nu\Vert^k_\infty)+\rho_{\mathrm{int}}(\Vert \omega-\hat \omega\Vert^k_\infty)\big]\Big).
\end{align}
	In view of Jensen's inequity and using equation \eqref{e:sf:1}, the convexity of $\alpha$ and the concavity of $\rho_{\mathrm{ext}},\rho_{\mathrm{int}}$, we have	
	\begin{align}\notag
		{\alpha}\Big(\EE \big[ \|\zeta_{a\nu\omega}(t)-\hat\zeta_{\hat a\hat\nu\hat\omega}(t)\|^k\big]\Big)  &\le \EE \left[ {\alpha}\left(\| \zeta_{a\nu\omega}(t)-\hat\zeta_{\hat a\hat\nu\hat\omega}(t)\|^k\right) \right] 
		 \le\EE \left[ V(\xi_{a\nu\omega}(t),\hat\xi_{\hat a\hat\nu\hat\omega}(t))\right]\\\notag &\le \vartheta\big(\EE[V(a,\hat a)], t\big)+
		 \eta^{-1}\Big(2\rho_{\mathrm{ext}}\big( \EE[\Vert \hat\nu\Vert^k_\infty]\big) + 2\rho_{\mathrm{int}}\big(\EE[\Vert \omega-\hat \omega\Vert^k_\infty]\big)\Big),
	\end{align}
	which in conjunction with the fact that $\alpha \in \Kinf$ leads to
	\begin{align*}
	\EE & \left[ \|\zeta_{a\nu\omega}(t)-\hat\zeta_{\hat a\hat\nu\hat\omega}(t)\|^k\right]
	\le{\alpha}^{-1} \bigg(\vartheta\big(\EE[V(a,\hat a)], t\big) + 
	\eta^{-1}\Big(2\rho_{\mathrm{ext}}\big( \EE[\Vert \hat\nu\Vert^k_\infty]\big) + 2\rho_{\mathrm{int}}\big(\EE[\Vert \omega-\hat \omega\Vert^k_\infty]\big)\Big)\bigg) 
	\\ & \le \alpha^{-1} \Big(2\vartheta\big(\EE[V(a,\hat a)], t\big) \Big) + \alpha^{-1}\Big(2\eta^{-1}\big(4\rho_{\mathrm{ext}}\big( \EE[\Vert \hat\nu\Vert^k_\infty]\big)\Big) 
	 + \alpha^{-1}\Big(2\eta^{-1}\big(4\rho_{\mathrm{int}}\big(\EE[\Vert \omega-\hat \omega\Vert^k_\infty]\big)\big)\Big).
	\end{align*}
	Therefore, by introducing functions $\beta$, $\gamma_{\mathrm{ext}}$, and $\gamma_{\mathrm{int}}$ as
	\begin{align}\notag
		\beta(r,t) &\Let {\alpha}^{-1}\big(2\vartheta(r,t)\big), 
		\\ \label{com-fn} \gamma_{\mathrm{ext}}(r) &\Let {\alpha}^{-1}\Big(2\eta^{-1}\big(4\rho_{\mathrm{ext}}(r)\big)\Big),
		\\ \notag \gamma_{\mathrm{int}}(r) &\Let {\alpha}^{-1}\Big(2\eta^{-1}\big(4\rho_{\mathrm{int}}(r)\big)\Big),
	\end{align}
	inequality (\ref{inequality}) is satisfied. Note that if $\alpha^{-1}$ and $\eta^{-1}$ satisfies the triangle inequality (i.e., $\alpha^{-1}(a+b)\leq\alpha^{-1}(a)+\alpha^{-1}(b)$ and $\eta^{-1}(a+b)\leq\eta^{-1}(a)+\eta^{-1}(b)$ for all $ a,b \in \R_{\geq0}$), one can divide all the coefficients by factor $2$ in the expressions of $\beta$, $\gamma_{\mathrm{ext}}$, and $\gamma_{\mathrm{int}}$ in \eqref{com-fn} to get a less conservative upper bound in \eqref{inequality}.
\end{proof}

\begin{proof}[Proof of Proposition \ref{lemma6}]
Since $V$ is an SSF-M$_k$ function from $\hat\Sigma$ to $\Sigma$ and $\eta(r)\geq\theta r$ for some $\theta\in\R_{>0}$ and any $r\in\R_{\geq0}$, for any
$\hat\nu\in\hat{\mathcal{U}}$, any $\hat \omega \in\mathcal{W}$, and any random variable $a$ and $\hat{a}$ that are $\sigalg_0$-measurable, there exists $\nu\in{\mathcal{U}}$ such that for all $\omega\in\mathcal{W}$ one obtains:
\begin{align*}
	\mathcal{L} V\left(\xi_{a\nu\omega}(t), \hat \xi_{\hat a\hat \nu\hat\omega}(t)\right) \leq&-\theta V\left(\xi_{a\nu\omega}(t), \hat \xi_{\hat a\hat \nu\hat\omega}(t)\right)+\rho_{\mathrm{ext}}(\Vert\hat \nu\Vert^k_\infty)+\rho_{\mathrm{int}}(\Vert \omega-\hat \omega\Vert^k_\infty).
\end{align*}
Since there exists a constant $\epsilon\geq0$ such that $\epsilon\geq\rho_{\mathrm{ext}}(\Vert\hat \nu\Vert^k_\infty)+\rho_{\mathrm{int}}(\Vert \omega-\hat \omega\Vert^k_\infty)$, one obtains:
\begin{align}\label{inequality10} 
	\mathcal{L} V\left(\xi_{a\nu\omega}(t), \hat \xi_{\hat a\hat \nu\hat\omega}(t)\right) \leq&-\theta V\left(\xi_{a\nu\omega}(t), \hat \xi_{\hat a\hat \nu\hat\omega}(t)\right)+\epsilon,
\end{align}
and the following chain of inequalities hold:
\begin{align}\nonumber
\PP\left\{\sup_{0\leq t\leq T}\left\Vert\zeta_{a\nu\omega}(t)-\hat \zeta_{\hat a\hat \nu\hat\omega}(t)\right\Vert\geq\varepsilon\,\,|\,\,[a;\hat a]\right\}=&\PP\left\{\sup_{0\leq t\leq T}\alpha\left(\left\Vert\zeta_{a\nu\omega}(t)-\hat \zeta_{\hat a\hat \nu\hat\omega}(t)\right\Vert^k\right)\geq\alpha(\varepsilon^k)\,\,|\,\,[a;\hat a]\right\}\\\label{final_probability}
\leq&\PP\left\{\sup_{0\leq t\leq T}V\left(\xi_{a\nu\omega}(t),\hat \xi_{\hat a\hat \nu\hat\omega}(t)\right)\geq\alpha(\varepsilon^k)\,\,|\,\,[a;\hat a]\right\}.
\end{align}
Using inequalities \eqref{inequality10}, \eqref{final_probability}, and Theorem 1 in \cite[Chapter III]{kushner}, one obtains the inequalities (\ref{simultaneous1}) and (\ref{simultaneous2}).
\end{proof}

\begin{proof}[Proof of Proposition \ref{lemma7}]
The proof is a simple consequence of Theorem \ref{theorem1} and Markov inequality \cite{oksendal}, used as the following:
\begin{align}
\PP\left\{\Vert\zeta_{a\nu\omega}(t)-\hat \zeta_{\hat a\hat \nu\hat\omega}(t)\Vert\geq\varepsilon\right\}&\leq\frac{\EE[\Vert\zeta_{a\nu\omega}(t)-\hat \zeta_{\hat a\hat \nu\hat\omega}(t)\Vert]}{\varepsilon}\leq\frac{\left(\EE\left[\Vert\zeta_{a\nu\omega}(t)-\hat \zeta_{\hat a\hat \nu\hat\omega}(t)\Vert^k\right]\right)^{\frac{1}{k}}}{\varepsilon}\\\notag\leq&\frac{\left(\beta\left(\EE[V(a,\hat a)],t\right )+ \gamma_{\mathrm{ext}}(\EE[\Vert\hat \nu\Vert^k_\infty]) + \gamma_{\mathrm{int}}(\EE[\Vert\omega-\hat \omega\Vert^k_\infty])\right)^{\frac{1}{k}}}{\varepsilon}.
\end{align}
\end{proof}

\begin{proof}[Proof of Corollary \ref{lemma8}]
Since $V$ is an SSF-M$_k$ function from $\hat\Sigma$ to $\Sigma$, for
$\hat\nu\equiv0$ and any random variable $a$ and $\hat{a}$ that are $\sigalg_0$-measurable, there exists $\nu\in{\mathcal{U}}$ such that one obtains:
\begin{align}\notag
	\mathcal{L} V\left(\xi_{a\nu}(t), \hat \xi_{\hat a0}(t)\right) \leq&-\eta \left(V\left(\xi_{a\nu}(t), \hat \xi_{\hat a0}(t)\right)\right),
\end{align}
implying that $V\left(\xi_{a\nu}(t), \hat \xi_{\hat a0 }(t)\right)$ is a nonnegative supermartingale \cite[Appendix C]{oksendal}. As a result, we have the
following chain of inequalities:
\begin{align}\nonumber
\PP\left\{\sup_{0\leq t< \infty}\left\Vert\zeta_{a\nu}(t)-\hat \zeta_{\hat a0 }(t)\right\Vert>\varepsilon\,\,|\,\,[a;\hat a]\right\}=&\PP\left\{\sup_{0\leq t< \infty}\alpha\left(\left\Vert\zeta_{a\nu}(t)-\hat \zeta_{\hat a0 }(t)\right\Vert^k\right)>\alpha(\varepsilon^k)\,\,|\,\,[a;\hat a]\right\}\\\notag\leq
&\PP\left\{\sup_{0\leq t< \infty}V\left(\xi_{a\nu}(t),\hat \xi_{\hat a0 }(t)\right)>\alpha(\varepsilon^k)\,\,|\,\,[a;\hat a]\right\}\leq\frac{V(a,\hat a)}{\alpha(\varepsilon^k)},
\end{align}
where the last inequality is implied from $V(\xi_{a\nu}(t),\hat \xi_{\hat a0 }(t))$ being a nonnegative supermartingale and \cite[Lemma1]{kushner}.
\end{proof}

\begin{proof}[Proof of Lemma \ref{lemma111}]
Consider the jump linear stochastic system $\Sigma$ with a linear feedback control law $u=Kx$, where $K\in\R^{m\times n}$, satisfying
\begin{equation}\notag
	\diff \xi(t)=(A+BK)\xi(t)\diff t+E\xi(t)\diff W_t + \sum_{i=1}^{\widetilde{q}}R_i\xi(t)\diff P^i_t. 	
\end{equation}
Define the matrix-valued deterministic process $\Phi(t) \Let \EE[\xi(t)\xi^T(t)]$. Applying the It\^{o}'s formula for jump diffusions \cite{ref:Oks-jump} leads to the following differential equations describing the time-evolution of the deterministic process $\Phi(t)$:
\begin{align}
\label{M}
&\dot{\Phi}(t) = \big(A+BK + \sum_{i=1}^{\widetilde{q}}\lambda_i R_i\big) \Phi(t) +  \Phi(t)\big(A +BK+ \sum_{i=1}^{\widetilde{q}}\lambda_i R_i \big)^T  + E\Phi(t)E^T + \sum_{i=1}^{\widetilde{q}}\lambda_i R_i \Phi(t)R_i^T.
\end{align}
To see further details on how the above ODE is derived, one can view each element of the matrix $\Phi(t)$ as an $\R$-valued mapping and treat it in the same way as we considered the Lyapunov function in the proof of Theorem~\ref{t:lin:suf}, and consequently arrives at \eqref{M}. From linear system theory, one can readily check that the ODE in \eqref{M} is asymptotically stable (implying $\Sigma$ is mean square asymptotically stable) if and only if $V(\Phi(t))=\text{Tr}(M\Phi(t))=\EE[\xi(t)^TM\xi(t)]$ is a Lyapunov function for \eqref{M} for a positive definite matrix $M$ satisfying condition \eqref{e:lin:con11}, which completes the proof.
\end{proof}

\begin{proof}[Proof of Lemma \ref{l:lin:con:DS}]
Suppose that $\im D\not\subseteq \im P+\im B$, then there exists $w\in
\R^p$ so that $Dw\neq P\hat x-B u$ holds for all $\hat x\in\R^{\hat
n}$, $u\in \R^m$. Hence~\eqref{e:lin:con2b} cannot hold for any
matrix $\hat D$ and $S$. Now suppose $\im D\subseteq \im P+\im B$. Let
$e_i$ denote the columns of $I_{p}$. Then there exist $\hat d_i\in\R^{\hat
n}$ and $s_i\in \R^m$ so that $De_i=P\hat d_i-Bs_i$ holds for all
$i\in\{1,\ldots,p\}$ and the matrices $\hat D=[\hat d_1\;\ldots\;\hat
d_p]$ and $S=[s_1\;\ldots\; s_p]$ satisfy~\eqref{e:lin:con2b}.
\end{proof}

\bibliographystyle{alpha}
\bibliography{refs,reference,ref_pey}

\end{document}